 \documentclass[final,3p,times,authoryear]{elsarticle}

\usepackage{amssymb}
\usepackage{color,latexsym,amsfonts,amssymb}
\usepackage{amsmath}
\usepackage{amsfonts,amssymb}
\usepackage{epsfig,epstopdf}
\usepackage{bm}
\usepackage{amsfonts,amsmath,latexsym}
\usepackage{mathrsfs}
\usepackage{amssymb,amsmath,bm}

\usepackage{ulem}

\usepackage{geometry}
\usepackage{caption}
\usepackage{enumerate}
\usepackage{url}
\usepackage{graphicx}

\usepackage{booktabs}
\usepackage{threeparttable}
\usepackage{multirow}
\usepackage{algorithm}
\usepackage{algorithmicx}

\usepackage{verbatim}

\newtheorem{lemma}{Lemma}[section]
\newtheorem{definition}{Definition}[section]

\newtheorem{theorem}{Theorem}[section]

\newtheorem{condition}{Condition}[section]

\def\blue{\color{blue}}
\def\proclaim#1{\par \bigskip\noindent {\bf #1}\bgroup\it\ }
\def\endproclaim{\egroup\par\bigskip}
\def\proof{\par\noindent{\bf Proof.} \;}
\def\proofthm3{\par\noindent{\bf Proof of Theorem 3.1.} \;}

\newbox\TempBox \newbox\TempBoxA

\def\ep{\mathbb{E}}
\def\var{\mathbb{V}ar}
\def\R{\mathbb{R}}
\def\I{\mathbb{I}}
\def\X{\mathbf{X}}
\def\x{\mathbf{x}}
\def\z{\mathbf{z}}
\def\P{\mathbb{P}}
\def\w{\mathbf{w}}
\def\W{\mathbf{W}}
\def\c{\mathbf{c}}
\def\bbeta{\boldsymbol{\beta}}
\def\bdelta{\boldsymbol{\delta}}

\def\bomega{\boldsymbol{\omega}}
\def\btheta{\boldsymbol{\theta}}
\def\bpi{\boldsymbol{\pi}}
\def\bW{\boldsymbol{W}}
\def\bPi{\boldsymbol{\Pi}}
\def\bxi{\boldsymbol{\xi}}
\def\bgamma{\boldsymbol{\gamma}}
\def\obeta{\hat{\bbeta}^*}
\def\oxi{\hat{\bxi}^*}
\def\sumn{\sum_{i=1}^{n}}

\newcommand{\bq}{\begin{equation}}
\newcommand{\eq}{\end{equation}}
\newcommand{\bqs}{\begin{equation*}}
\newcommand{\eqs}{\end{equation*}}
\newcommand{\be}{\begin{eqnarray}}
\newcommand{\ee}{\end{eqnarray}}
\newcommand{\by}{\begin{eqnarray*}}
\newcommand{\ey}{\end{eqnarray*}}
\newcommand{\bn}{\begin{enumerate}}
\newcommand{\en}{\end{enumerate}}
\newcommand{\bi}{\begin{itemize}}
\newcommand{\ei}{\end{itemize}}
\newcommand{\bds}{\begin{description}}
\newcommand{\eds}{\end{description}}
\newcommand{\bcen}{\begin{center}}
\newcommand{\ecen}{\end{center}}

\journal{}

\begin{document}

\begin{frontmatter}


\title{ Semiparametric Expectile Regression for High-dimensional Heavy-tailed and Heterogeneous Data
	\tnoteref{label_title}}
\tnotetext[label_title]{This research is partly supported by the Fundamental Research Funds for the Central Universities, Major Project of the National Social Science Foundation of China (No.13$\&$ZD163) , Zhejiang Provincial Natural Science Foundation (No: LY18A010005) and the Research Project of Humanities and Social Science of Ministry of Education of China(No. 17YJA910003). Jun Zhao and Guan'ao Yan are co-first authors.}

\author[Label1]{Jun Zhao}

\author[Label2]{Guan'ao Yan}

\author{Yi Zhang\corref{cor1}\fnref{Label2}}
\address[Label1]{Zhejiang University City College}
\address[Label2]{School of Mathematical Sciences, Zhejiang University}
\cortext[cor1]{ Address for correspondence: Yi Zhang, School of Mathematical Sciences, Zhejiang University, 38 Zheda Road, Hangzhou, 310027, Zhejiang, China.}
\ead{zhangyi63@zju.edu.cn}

\begin{abstract}
Recently, high-dimensional heterogeneous data have attracted a lot of attention and discussion. Under heterogeneity, semiparametric regression is a popular choice to model data in statistics. In this paper, we take advantages of expectile regression in computation and analysis of heterogeneity,
 and propose the regularized partially linear additive expectile regression with nonconvex penalty, for example, SCAD or MCP for such high-dimensional heterogeneous data. We focus on a more realistic scenario: the regression error is heavy-tailed distributed and only has finite moments, which is  violated with the classical sub-gaussian distribution assumption and more common in practise.
Under some regular conditions, we show that with probability tending to one, the oracle estimator is one of the local minima of our optimization problem. The theoretical study indicates that the dimension cardinality of linear covariates our procedure can handle with is essentially restricted by the moment condition of the regression error.
For computation, since the corresponding optimization problem is nonconvex and nonsmooth, we derive a two-step algorithm to solve this problem.
Finally, we demonstrate that the proposed method enjoys good performances in estimation accuracy and model selection through Monto Carlo simulation studies and a real data example.  What's more, by taking different expectile weights  $\alpha$, we are able to detect heterogeneity and explore the entire conditional distribution of the response variable, which indicates the usefulness of our proposed method for analyzing high dimensional heterogeneous data.

\end{abstract}


\end{frontmatter}


\section{Introduction}
\label{intro}
\setcounter{definition}{0}\setcounter{definition}{0}
\setcounter{equation}{0}\setcounter{lemma}{0}
\setcounter{proposition}{0}\setcounter{theorem}{0}
\setcounter{remark}{0}\setcounter{corollary}{0}

High dimensional data have been increasingly frequent in nowadays scientific areas like genomics, economics, finance and so on. The past two decades have witnessed the rapid development of high dimensional statistical analysis,  most of which usually assume homogeneity on the data structure, for example, assuming a sequence of i.i.d. errors in regression framework(see, \cite{BaV11}). Recently, however, \cite{NRC13} points out that due to multi-sources data collection technology and error accumulation in data preprocessing, high dimensional data display an opposite feature "heterogeneity". There has been much evidence for the existence of heterogeneity in high dimensional data, for example, \cite{Da12} identified the presence of heteroscedasticity in the eQTLs data which are usually associated with gene expression variations and demonstrated directly the necessity of considering heteroscedasticity in the modeling; \cite{WWL12} used regularized quantile regression to investigated genetic variation relevant to human eye disease and also detected heterogeneity in this gene data; other references, see \cite{GaZ16},~\cite{ZCZ18}.

 \cite{Buja14} pointed out that  when some covariates indeed cast nonlinear effects on the response, a blind use of linear approximations may lead to misunderstanding in data mining. From this point of view, it needs to incorporates the nonparametric effect in the modelling process in some circumstances. Semiparametric regression absorbs the simplicity  of a linear form and the flexibility of nonparametric models so that it has been widely used to model heterogeneous data in statistics and econometrics, see, \cite{RaS96} and \cite{Horo99}.  Meanwhile, when building up the model, additivity is often assumed on the nonparametric part to circumvent the curse of dimensionality (\cite{HaT90},~\cite{Gyo06}). Thus, a partially linear additive model is often applied as an intermediate strategy to make the analysis more reliable and flexible.
Under this framework, \cite{SaW16} proposed the regularized partially linear additive quantile regression to analyze high-dimensional heterogeneous data. However, the non-differential characteristic of the quantile loss function may impose restriction on the computation efficiency in high dimension. Besides, quantile regressions at different levels are needed to account for heterogeneity. So crossing problems may happen and cause misunderstanding when B-spline approximation is introduced in the nonparametric part of semiparametric regression. 
These observations motivate us to use expectile regression proposed by  \cite{AAP76} and \cite{New87} to deal with heterogeneity.
 Expectile regression uses the asymmetric squares loss function $\phi_{\alpha}(\cdot)$ defined below,
\be\label{phi}
\phi_{\alpha}(r)=|\alpha-\I(r<0)|r^2=
\begin{cases}
\alpha r^2, &r \geq 0, \\
(1-\alpha)r^2, &r<0.
\end{cases}
\ee
And the $\alpha$-th expectile of random variable $y$ is denoted by
\bqs m_{\alpha}(y)=\underset{m\in\R}{\arg\min}~ \ep\phi_{\alpha}(y-m).\eqs
Note that 1/2-th expectile is exactly the mean.
Expectile regression shares good properties. Firstly, its differentiable loss function can alleviate the computation burden and further leads to a more amenable theoretical process under the high dimensional setting. Furthermore, assigning different weights onto squared error loss according to the positive and negative error respectively, we can infer a complete interpretation of the conditional relationship between covariates and response variable, see  \cite{New87}. What's more, \cite{Wal15} conclude from their simulation results that expectile regression seems less vulnerable for crossing problems than quantile regression and brings some kind of robustness in nonparametric approximation.
Due to these good properties, expectile regression has great potential to analyze heterogeneity in semiparametric framework.

Meanwhile, most of the existing studies in the literature assume that regression errors follow the
gaussian or sub-gaussian distribution. For example, under this assumption,
\cite{GaZ16} proposed the linear expectile regression model to analyze heteroscedasticity in high dimension. However, this is a relatively stringent  assumption and more and more evidence is against it, especially in genetics and finance, where regression errors do not have the tail of exponential decreasing rate (\cite{FLW17}) or even worse is heavy-tailed with only finite moments (\cite{ZCZ18}).

In this article, we develop the methodology and theory for the partially linear additive expectile regression with general nonconvex penalty function for  high-dimensional heterogeneous data. As suggested in \cite{Sch07}, the linear part of our model incorporates most variables of the data set and its dimensionality $p=p(n)$ can be much larger than the sample size $n$. As for the nonparametric part, it contains some explaining variables, like some clinical or environmental variables used for describing their possible nonlinear effects and its dimensionality $d$ is fixed. To account for heterogeneity, we adopt the so-called variance heterogeneity in \cite{RaS96} and let regression errors allow for either nonconstant variances or other forms of covariate effects. More importantly, in our framework, unlike the classical gaussian or sub-gaussian errors distribution,  we make a less stringent and more realistic assumption that regression errors only have finite moments. Theoretically, we derive the asymptotic oracle properties for the optimal estimator of our proposed regularized expectile regression framework and figure out how heavy-tailed moment condition affects the dimensionality of covariates our model can handle. For computation, since the corresponding optimization problem is nonconvex and nonsmooth, we make full use of the structure of our induced optimization  problem and seperate it into the penalized linear part and the nonparametric part, and propose a two step algorithm. In each step, we take full advantage of expectile regression with its differentiability to reduce computation burden.

This article is organized as follows. In Section 2, we introduce our penalized partially linear additive expectile regression with nonconvex penalty functions like SCAD or MCP penalty and provide an efficient algorithm for the optimization problem. In Section 3, we propose the oracle estimator as a benchmark and  build up its relation with our induced optimization problem, i.e., the so-called oracle property.
In Section 4, we carry out the Monte Carlo simulation to investigate performances of the proposed method under the heteroscedasticity setting. In Section 5, we apply our model onto a genetic microarrays data set to analyze the potential inducement for low infant birth weights. The proofs of theoretical results and lemmas needed are included in the Appendix.



\section{Methodology}
\label{method}
\setcounter{definition}{0}\setcounter{definition}{0}
\setcounter{equation}{0}\setcounter{lemma}{0}
\setcounter{proposition}{0}\setcounter{theorem}{0}
\setcounter{remark}{0}\setcounter{corollary}{0}
\subsection{Partially linear additive expectile regression}
Semiparametric regression incorporates most variables of the data set into the linear part, and contains some explaining variables like some clinical or environmental variables in the nonparametric part  to describe possible nonlinear effects casted on responses.
Let us begin with the notation and statistical setup. Suppose that we have a high-dimensional data sample $\{Y_i,\x_i,\z_i\}$,~$i=1,\ldots,n$, where $\x_i=(x_{i1},\ldots,x_{ip}),~i=1,\ldots,n$ are independent and identically distributed $p$-dimensional covariates along with the common mean 0 and $\z_i=(z_{i1},\ldots,z_{id}),~i=1,\ldots,n$ are $d$-dimensional covariates. Consider the data are from the following partially linear additive model,
\bq\label{high-dimensional partially linear model}
Y_i=\mu_0+\sum_{k=1}^p\beta_k^*x_{ik}+\sum_{j=1}^dg_{0j}(z_{ij})+\epsilon_i=\x_i'\boldsymbol{\beta}^*+g_0(\z_i)+\epsilon_i,
\eq
where $g_0(\z_i)=\mu_0+\sum_{j=1}^dg_{0j}(z_{ij})$.

In consideration of data heterogeneity, here we allow for the so-called "variance heterogeneity" from \cite{RaS96} in their "mean and dispersion additive model". For example,  $\epsilon_i$ can take the following form, $$\epsilon_i=\sigma(\x_i,\z_i)\eta_i,$$ where $\sigma(\x_i,\z_i)$, conditional on $\x$ and $\z$, can be nonconstant, linear form (\cite{Ait87},~\cite{GaZ16}), nonparametric form (\cite{RaS96}).

In addition to heterogeneity, $\{\epsilon_i\}_{i=1}^n$ are  assumed to be mutually independent and satisfy $m_{\alpha}(\epsilon_i|\x_i,\z_i)=0$ for some specific $\alpha$. Thus, by $m_{\alpha}(\epsilon_i|\x_i,\z_i)=0$, the $\alpha$th conditional expectile of $Y_i$ given covariates $\x_i,\z_i$ through the model (\ref{high-dimensional partially linear model}) is
\bqs
m_{\alpha}(Y_i|\x_i,\z_i)=\x_i'\boldsymbol{\beta}^*+g_0(\z_i).
\eqs
So $\bbeta^*$ and $g_0(\cdot)$ actually minimize the conditional $\ep[\phi_{\alpha}(Y_i-\x'\bbeta-g(\z))]$,\footnote{For notation simplicity, when no confusion arises, we still denote conditional expectation by $\ep[\cdot|\x,\z] := \ep[\cdot]$. }
\be
(\bbeta^*,g_0(\z))=\underset{\bbeta\in\R^p,g\in \mathcal{G}}{\arg\min} \ep[\phi_{\alpha}(Y_i-\x'\bbeta-g(\z))].
\ee
 For identification purpose, without loss of generality, each $g_{0j}$ is assumed to has zero mean, i.e., $g_0(\z)\in \mathcal{G}$, where
\bqs
\mathcal{G}=\{g(\z):g(z)=\mu+\sum_{j=1}^d g_j(z_j), \ep[g_j(z_j)]=0,~j=1,\ldots,d\},
\eqs
so that the minimizer $(\bbeta^*,g_0(\z))$  of the population risk is unique for convenience.

In this article, $d$, the dimensionality of the nonparametric covariates $\z$,  is fixed.
The nonparametric components $g_{0j}(\cdot), j=1,\ldots,d$ in this model are approximated by a linear combination of B-spline basis functions. Before that, we first give out the following definition:
\begin{definition}
Let $r\equiv m+\nu$, where $m$ is a positive integer and $\nu\in (0,1]$. Define $\mathcal{H}_r$ as the collection of functions $h(\cdot)$ on $[0,1]$ whose $m$th derivative $h^{(m)}(\cdot)$ satisfies the H\"{o}lder condition of order $\nu$. That is, for any $h(\cdot)\in \mathcal{H}_r$, there exists some positive constant $C$ such that
\be
|h^{(m)}(z')-h^{(m)}(z)|\leq C|z'-z|,~~\forall ~~0\leq z', z\leq 1.
\ee
\end{definition}

Assume the nonparametric component $g_{0j}(\cdot) \in \mathcal{H}_r$ for some $r\geq 1.5$. Let $\bpi(t)=(b_1(t),\ldots,b_{k_n+l+1}(t))'$ denote a vector of normalized B-spline basis functions of order $l+1$ with $k_n$ quasi-uniform internal knots on $[0,1]$. Then $g_{0j}(\cdot)$ can be approximated using a linear combination of B-spline basis functions in the $\bPi(\z_i)=(1,\bpi(z_{i1})',\ldots,\bpi(z_{id})')'$. Then there exists $\bxi_0=(\xi_{00},\bxi_{01},\ldots,\bxi_{0d})\in \R^{D_n}$, where $D_n=d(k_n+l+1)+1$, such that $\sup_{\z_i}|\bPi(\z_i)'\bxi_0-g_0(\z_i)|=O(k_n^{-r})$; see \cite{Sto85},~ \cite{Sch07}. For ease of notation and simplicity of proofs, we use the same number of basis functions for all nonlinear components in model (\ref{high-dimensional partially linear model}), but in practice, such restrictions are not necessary.

The dimensionality of $\x$, $p=p(n)$, follows high-dimensional setting, and is much larger than $n$. One leading way to deal with this high dimension setting is to assume that the true parameter $\boldsymbol{\beta}^{*}=(\beta_1^{*},\ldots,\beta_p^{*})$ is  sparse. Let $A=\{j: \beta_j^{*}\neq0,1\leq j\leq p\}$ be the active index set and its cardinality $q=q(n)=|{A}|$. Sparsity means that $q<n$ and all the left $(p-q)$ coefficients are exactly zero. Generally, due to heterogeneity,  $\boldsymbol{\beta}^{*}$, $A$ and $|{A}|$ can change for different expectile level $\alpha$ and for simplicity in notation, we omit such dependence when no confusion arises. Without loss of generality, we rewrite $\boldsymbol{\beta}^*=((\boldsymbol{\beta}_A^*)',\mathbf{0}')'$ where $\boldsymbol{\beta}_A^*\in \R^q$ and $\mathbf{0}$ denotes a $(p-q)$ dimensional vector of zero. Let $\X=(\x_1,\ldots,\x_n)'$ be the $n\times p$ matrix of covariates. Denote $\X_j$ the $j$th column of $\X$ and define $\X_A$ the submatrix of $\X$ that consists of its first $q$ columns and denote by $\X_{A_i}$ the ith row of $\X_A$.

Under sparsity assumption,
regularized framework has been playing a leading role in analyzing high-dimensional data in the past two decades.
There are different lines of choices for the penalty function $P_{\lambda}(t)$ . The $L_1$ penalty, the well-known Lasso (\cite{Tib96}), is a popular choice for penalized estimation since it induces a convex optimization problem such that it brings convenience in theoretical analysis and computation. However, the $L_1$ penalty is known to over-penalize large coefficients, tends to be biased and requires strong irrepresentable conditions on the design matrix to achieve selection consistency. This is usually not a concern for prediction, but can be undesirable if the goal is to identify the underlying model. In comparison, an appropriate nonconvex penalty function can effectively overcome this problem; see \cite{Fan01}. So throughout this paper, we assume that the regularizer $P_{\lambda}(t)$ is a general folded concave penalty function, for examples, the SCAD or MCP penalty;
	
\bi
\item \textbf{SCAD, \cite{Fan01}}. ~The SCAD penalty is defined through its first order derivative and symmetry around the origin. To be specific, for $\theta > 0$,
\be
P_{\lambda}'(\theta)=\lambda\{\I(\theta\leq\lambda)+\frac{(a\lambda-\theta)_+}{(a-1)\lambda}\I(\theta>\lambda)\},
\ee
where $a> 2$ is a fixed parameter. By straight calculation, for $\theta > 0$,
 \be
 P_{\lambda}(\theta)&=&\lambda\theta\I(\theta\leq\lambda)+\frac{a\lambda\theta-(\theta^2+\lambda^2)/2}{a-1}\I(\lambda\leq\theta\leq a\lambda)\nonumber\\
 &&+\frac{(a+1)\lambda^2}{2}\I(\theta> a\lambda).
 \ee
\item \textbf{ MCP, \cite{Zha10}}.~The MCP penalty has the following form:
\be
P_{\lambda}(\theta)=\text{sgn}(\theta)\lambda\int_{0}^{|\theta|}(1-\frac{z}{\lambda b})dz,
\ee
where $b>0$ is a fixed parameter and $\text{sgn}(\cdot)$ is the sign function.
\ei

From the definition above,  the SCAD or MCP penalty function is
symmetric,non-convex on $[0,\infty)$ , and singular at the origin. $a=3.7$ and $b=1$ are suggested as a practical choice for the SCAD or MCP penalty respectively for good practical performance in various variable selection problems.

The proposed estimators are obtained by solving the following optimization problem,
\bq\label{non-convex regularized problem}
(\hat{\boldsymbol{\beta}},\hat{\bxi})=\underset{\boldsymbol{\beta} \in \R^{p},\bxi\in\R^{D_n}}{\arg\min}~L(\boldsymbol{\beta},\bxi),
\eq
where $L(\boldsymbol{\beta},\bxi)$, the penalized expectile loss function for our model, is
\bq\label{penalized expectile loss function}
L(\boldsymbol{\beta},\bxi)=\frac{1}{n}\sum_{i=1}^n\phi_{\alpha}(y_i-\x_i'\boldsymbol{\beta}-\bPi(\z_i)'\bxi)+\sum_{j=1}^pP_{\lambda}(|\beta_j|).
\eq
Denote by $\hat{\bxi}=(\hat{\xi}_0,\hat{\bxi}_1,\ldots,\hat{\bxi}_d)$, then the estimator of $g_{0}(\z_i)$ is
\by \hat{g}(\z_i)&=&\hat{\mu}+\sum_{j=1}^d\hat{g}_j(z_{ij}), \ey where
\by
\hat{\mu}&=&\hat{\xi}_0+n^{-1}\sum_{i=1}^n\sum_{j=1}^d \bpi(z_{ij})'\hat{\bxi}_j,\\
\hat{g}_j(z_{ij})&=&\bpi(z_{ij})'\hat{\bxi}_j-n^{-1}\sum_{i=1}^n \bpi(z_{ij})'\hat{\bxi}_j.
\ey
The centering above is just the sample analog of the identifiability assumption  $\ep[g_{0j}(\z_{j})]=0$ for $j=1,\ldots,d$.

\subsection{Algorithm}
For the optimization problem (\ref{non-convex regularized problem}), note that
there is no penalty on the nonparametric coefficients $\bxi$. So instead of taking $(\boldsymbol{\beta},\bxi)$ as the whole optimization parameters, we decompose the optimization problem into two parts respectively: the fixed dimensional unpenalized nonparametirc part and the high dimensional penalized linear part.  An iterative two-step algorithm is proposed by us.

 To be specific, in the first step, we obtain the nonlinear part's parameters by minimizing an unpenalized objective function in $D_n$ dimension with the parameters from the linear part valued at its previous-iteration result. Note  by Lemma \ref{Prop of Loss} that, the expectile loss function $\phi_{\alpha}(\cdot)$ is differentiable and strongly convex, this optimization problem can be easily done by convex analysis. Then after solving the nonparametric optimization problem, in the second step, we obtain the linear part's parameters by minimizing a penalized expectile loss function. Due to the non-convexity of the penalty, we have to deal with a non-convex optimization problem in high dimension. Here we take use of Local Linear Approximation (LLA,~\cite{ZaL08}) strategy to approximate the penalized optimization problem into a convex one, due to its computational efficiency and good statisitical properties, see \cite{FXZ14}. The detailed algorithm is as follows,
\begin{algorithm}[H]
	\caption{The two-step algorithm for  the nonconvex optimization problem (\ref{non-convex regularized problem})}
	\label{LLA algorithm}
	\begin{algorithmic}[1]
		\State Initialize $\boldsymbol{\beta}^{(0)}=\boldsymbol{\beta}^{\text{initial}}$.

\State For $t=1,2,\ldots$, repeat the following iteration (a) and (b) until convergence
	
	\bn[(a)]
		\item \textbf{The nonparametric part}: At $t$-th iteration,
		based on the previous solution $\bbeta^{(t-1)}$, $\bxi^{(t)}$ is obtained by minimizing the following function
		\bqs
		\bxi^{(t)}=\underset{\bxi\in \R^{D_n}}{\arg\min } \frac{1}{n}\sum_{i=1}^n\phi_{\alpha}(y_i-\x_i'\boldsymbol{\beta}^{(t-1)}-\bPi(\z_i)'\bxi).
		\eqs

		\item \textbf{The linear part}: Then next, at $t$-th iteration, $\boldsymbol{\beta}^{(t)}$ is obtained by the following procedure,

		\bn[(b.1)]
		\item At $t$-th iteration, calculate the corresponding weights based on the current solution $\bbeta^{(t-1)}=(\beta_1^{(t-1)},\ldots,\beta_p^{(t-1)})'$
		\bqs \bomega^{(t)}=(\omega_1^{(t)},\ldots,\omega_p^{(t)})'=(P_{\lambda}'(|\beta_1^{(t-1)}|),\ldots,P_{\lambda}'(|\beta_p^{(t-1)}|))'.
		\eqs
		\item the current local linear approximation of  regularized loss function $L(\boldsymbol{\beta},\bxi^{(t)})$, denoted by $L(\boldsymbol{\beta}|\boldsymbol{\beta}^{(t-1)},\bxi^{(t)})$, is
		\bqs
		L(\boldsymbol{\beta}|\boldsymbol{\beta}^{(t-1)},\bxi^{(t)})=\frac{1}{n}\sum_{i=1}^n\phi_{\alpha}(y_i-\x_i'\boldsymbol{\beta}-\bPi(\z_i)'\bxi^{(t)})+\sum_{j=1}^p\omega_j^{(t)}|\beta_j|.
		\eqs
		
		\item $\boldsymbol{\beta}^{(t)}=\underset{\boldsymbol{\beta}\in \R^{p}}{\arg\min}~ L(\boldsymbol{\beta}|\boldsymbol{\beta}^{(t-1)},\bxi^{(t)})$.
		\en
		\en
	\end{algorithmic}
\end{algorithm}

The initial value $\boldsymbol{\beta}^{(0)}$ can be chosen as the estimator from the penalized expectile regression without the nonparametric part. In each step,
the involved optimization problems in the algorithm above are convex after modification, and taking full advantage of expectile regression with its differentiability, there are many powerful programs to solve them. For example, to solve problem (b.3), we can apply the proximal gradient method, and use CVX, a Matlab package for specifying and solving convex programs; see \cite{MaS08},~\cite{MaS13}.

\section{Asymptotic theory}
\setcounter{definition}{0}\setcounter{definition}{0}
\setcounter{equation}{0}\setcounter{lemma}{0}
\setcounter{proposition}{0}\setcounter{theorem}{0}
\setcounter{remark}{0}\setcounter{corollary}{0}

\subsection{Oracle Study}
If we could know what variables were significant in advance,
\cite{Fan01} proposed the so-called oracle estimator as a performance benchmark. Following their idea, we fisrt introduce the oracle estimator for partially linear additive model, denoted by
$(\hat{\bbeta}^*,\hat{\bxi}^*)$
with $\hat{\bbeta}^*=(\hat{\bbeta}_A^{*'},\mathbf{0}_{p-q}^{'})^{'}$ through the following optimization problem:
\be\label{Oracle Estimator Representation}
(\hat{\bbeta}_A^*,\hat{\bxi}^*)=\underset{\bbeta\in \R^{q},\bxi\in \R^{D_n}}{\arg\min}\frac{1}{n}\sum_{i=1}^n\phi_{\alpha}(y_i-\x_{A_i}'\bbeta-\bPi(\z_i)'\bxi).
\ee

The cardinality $q_n$ of the index set $A$ is allowed to change with $n$ so that a more complex statistical model can be fit when more data are collected. This setup violates with the classical scenario (\cite{HaT90}) where the cardinality is fixed.  So we need to investigate the asymptotic property of the oracle estimator above.

Denote by the derivative of $\phi_\alpha(r)$ as $\psi_\alpha(r)$, i.e.,
\bqs \psi_\alpha(r) = 2|\alpha-\I(r<0)|r. \eqs
Notice that  $\phi_\alpha(r)$ does not have second-order derivative at $r=0$, thus we define an analog second-order derivative as $\varphi_\alpha(r)$ as
\bqs
\varphi_\alpha(r) = \begin{cases}
	2|\alpha-\I(r<0)|,\, & r \neq 0, \\
	\in 2[c_1,c_2], & r=0,
\end{cases}
\eqs
where $c_1\triangleq\min\{\alpha,1-\alpha\}$ and $c_2 \triangleq \max\{\alpha,1-\alpha\}$.

For theoretical study purpose, let $w_i=\ep[\varphi_\alpha(\epsilon_i)|\x_i,\z_i]$ and we consider the following weighted projetion from $\x$ onto $\z$,
\be
 h^*_j(\cdot) = \underset{h_j(\cdot)\in\mathcal{G}\cap \mathcal{H}_r }{\arg\inf}\sumn \ep[w_i\cdot (x_{ij}-h_j(\z_i))^2],
\ee
This projection strategy is commonly used  in the semiparametric analysis, see, \cite{Robin88}, \cite{DaN94} and \cite{LaL09}.
Define $m_j(\z) = \ep[x_{ij}|\z_i=\z]$. Then, $h^*_j(\z)$ is a weighted projection from $m_j(\z)$ to $\mathcal{G}\cap \mathcal{H}_r$ under $L_2$ norm, where the weights $w_i$ accommodate for possilbe heterogeneity. Then we define $H = (h^*_j(\z_i) )_{n\times q_n}$, $\delta_{ij} = x_{A_{ij}} - h^*_j(\z_i)$, the vector $\bdelta_i = (\delta_{i1},\ldots,\delta_{iq_n})^{'} \in \R^{q_n}$ and the matrix $\Delta_n = (\bdelta_1,\ldots,\bdelta_n)^{'} \in \R^{n\times q_n}$. Thus, $X_A = H + \Delta_n$.

 Before we give out our results, we need some technical conditions.
\begin{condition}\label{C(1)}
	There exists a positive constant $C$ such that  $\ep\left( \epsilon_{i}^{2k}|x_{i},\z_i \right) < C < \infty$ for all $ i$ and some $k \geq 1$.
\end{condition}

\begin{condition}\label{C(2)}
	There exists positive constants $M_1$ and $M_2$ such that $|x_{ij}| \leq M_1,~\forall1 \leq i \leq n,~1\leq j\leq p_{n}$ and $\ep\left( \delta_{ij}^{4}\right) \leq M_2,~\forall1 \leq i \leq n,~1\leq j\leq q_{n}$. There exist finite positive constants $C_1$ and $C_2$ such that with probability one
	\by
	C_1 \leq \lambda_{max}\left(n^{-1}X_{A}X_{A}^{'}\right) \leq C_2,~~
	C_1 \leq \lambda_{max}\left(n^{-1}\Delta_{n}\Delta_{n}^{'}\right) \leq C_2.
	\ey
\end{condition}
\begin{condition}\label{C(3)}
	For $r = m+v > 1.5$, $g_0(\cdot) \in \mathcal{G}\cap \mathcal{H}_r$. The dimension of the spline basis $k_n$ satisfies $k_n \approx n^{1/(2r+1)}$
\end{condition}
\begin{condition}\label{C(4)}
	$q_n = O(n^{C_3})$ for some $C_3 < \frac{1}{2}$.
\end{condition}

Condition 3.1 does not require the error sequences follow i.i.d. assumption, and what's more, the imposed moment condition
on the random error is more relaxed than the classical Gaussian or sub-Gaussian tail condition. This condition is also used in \cite{Kim08},~\cite{ZCZ18}. Condition 3.2 is about the behavior
of the covariates and the design matrix under the oracle model, which is not
restrictive. Condition 3.3 is typical for the application of B-splines, which guarantees the approximation accuracy and convergence rate of $\hat{g}(\cdot)$, see \cite{Sto85}. Condition 3.4 is about the sparsity level and standard for linear models with diverging number of parameters.
\begin{theorem}\label{oracle estimator asymptotic property}
	Assume conditions 3.1-3.4 hold. Then the oracle estimator obtained by the optimization problem (\ref{Oracle Estimator Representation}) satisfies
	\be
	\parallel\hat{\bbeta}^*_A-\bbeta_A\parallel&=&O_p(\sqrt{n^{-1}q_n}),\\
	n^{-1}\sum_{i=1}^{n}(\hat{g}(\z_i)-g_0(\z_i))^2&=&O_p(n^{-1}(q_n+k_n)).
	\ee
\end{theorem}
\subsection{Differencing Convex procedure }
 Since we introduce the nonconvex penalty to achieve estimation bias reduction and model consistency in our regularized framework, the objective function $L(\boldsymbol{\beta},\bxi)$ is nonconvex. Furthermore, although the asymmetric least squares loss $\phi_{\alpha}(r)$ defined in (\ref{phi}) is differentiable $\forall r \in \R $, $\phi_{\alpha}(r)$ is not smooth due to non-existence of its second order derivative at the point $r=0$. So the classical KKT condition appears not applicable in this situation to analyze the asymptotic properties of the penalized estimator from the optimization problem (\ref{non-convex regularized problem}).

  Suppose the SCAD penalty is used. Note that $L(\boldsymbol{\beta},\bxi)$ can be decomposed into the difference of two convex functions,
 \bqs
L(\boldsymbol{\beta},\bxi)= k(\boldsymbol{\beta},\bxi)- l(\boldsymbol{\beta},\bxi)
\eqs
 where the two convex functions are given below,
 \footnote{The specific fomula of $H_{\lambda}(\theta)$ depends on the penalty used. Here we choose the SCAD penalty as an example and in fact, the MCP penalty also shares the similar form and leads to the decomposition property of the objective function $L(\boldsymbol{\beta},\bxi)$. }
 \by
k(\boldsymbol{\beta},\bxi)&=&\frac{1}{n}\sum_{i=1}^n\phi_{\alpha}(y_i-\x_i'\boldsymbol{\beta}-\bPi(\z_i)'\bxi)+\lambda\sum_{j=1}^p|\beta_j|,\\
 l(\boldsymbol{\beta},\bxi)&=&\sum_{j=1}^p H_{\lambda}(\beta_j),\,~~~\text{where}\\
 H_{\lambda}(\theta)&=&[(\theta^2-2\lambda|\theta|+\lambda^2)/(2(a-1))]\I(\lambda\leq|\theta|\leq a\lambda)\\
 &&+[\lambda|\theta|-(a+1)^2/2]\I(|\theta|>a\lambda).
 \ey
 Moreover, $H_{\lambda}(\theta)$ is differentiable everywhere,
 \bqs
 H'_{\lambda}(\theta)=[(\theta-\lambda \text{sgn}(\theta))/(a-1)]\I(\lambda\leq|\theta|\leq a\lambda)+\lambda \text{sgn}(\theta)\I(|\theta|>a\lambda).
 \eqs

\cite{TaA97} studied such nonconvex optimization problem where the objective function can be expressed as the difference of two convex functions and gave out  the sufficient condition for the local minimizer, see Lemma \ref{Sufficient condition for local optimal minimizer} for detail.
Therefore, this result in \cite{TaA97} provides us an available approach to exploring the sufficient condition for the local minimizer of the nonconvex optimization problem induced by regularized partially linear additive expectile regression.
Before doing this, we introduce some notations.
Denote by the unpenalized empirical loss function,
\be
L_n(\boldsymbol{\beta},\bxi)=\frac{1}{n}\sum_{i=1}^n\phi_{\alpha}(y_i-\x_i'\boldsymbol{\beta}-\bPi(\z_i)'\bxi).
\ee
 $L_n(\boldsymbol{\beta},\bxi)$ is differentiable with respect to $\bbeta$ and $\bxi$.
Denote by for $j=1,\ldots, p$,
\by
s_j(\bbeta,\bxi)&=&\frac{\partial}{\partial \beta_j}(\frac{1}{n}\sum_{i=1}^n\phi_{\alpha}(y_i-\x_i'\bbeta-\bPi(\z_i)'\bxi))\nonumber\\
&=&-\frac{2}{n}\sum_{i=1}^n\alpha\x_{ij}(y_i-\x_i'\bbeta-\bPi(\z_i)'\bxi)\I(y_i-\x_i'\bbeta-\bPi(\z_i)'\bxi\geq0)\nonumber\\
&&-\frac{2}{n}\sum_{i=1}^n(1-\alpha)\x_{ij}(y_i-\x_i'\bbeta-\bPi(\z_i)'\bxi)\I(y_i-\x_i'\bbeta-\bPi(\z_i)'\bxi<0),
\ey
and $j=p+l,~l=1,\ldots,D_n$,
\by
s_j(\bbeta,\bxi)&=&\frac{\partial}{\partial \xi_l}(\frac{1}{n}\sum_{i=1}^n\phi_{\alpha}(y_i-\x_i'\bbeta-\bPi(\z_i)'\bxi))\nonumber\\
&=&-\frac{2}{n}\sum_{i=1}^n\alpha\bPi_l(\z_i)(y_i-\x_i'\bbeta-\bPi(\z_i)'\bxi)\I(y_i-\x_i'\bbeta-\bPi(\z_i)'\bxi\geq0)\nonumber\\
&&-\frac{2}{n}\sum_{i=1}^n(1-\alpha)\bPi_l(\z_i)(y_i-\x_i'\bbeta-\bPi(\z_i)'\bxi)\I(y_i-\x_i'\bbeta-\bPi(\z_i)'\bxi<0),
\ey
where $\bPi(\z_i)=(1,\Pi_1(\z_i),\ldots,\Pi_{L_n}(\z_i))$ is the basic function at $\z_i$.

Note that $k(\cdot)$ is not differentiable due to the $L_1$ penalty. Define the subdifferential of $k(\btheta)$ at $\btheta=\btheta_0$ as follows:
\bqs
\partial k(\btheta_0)=\{t:k(\btheta)\geq k(\btheta_0)+t'(\btheta-\btheta_0), \forall \btheta\}.
\eqs
 Therefore,  $\partial k(\boldsymbol{\beta},\bxi)=\left\{\boldsymbol{\kappa}=(\kappa_1,\ldots,\kappa_{p+D_n})'\in \R^{p+D_n}\right\}$ has the following expression,
\by
\kappa_j=
\begin{cases}
	s_j(\bbeta,\bxi)+\lambda l_j, &j=1,2,\ldots,p, \\
	s_j(\bbeta,\bxi), &j=p+1,\ldots,p+D_n;
\end{cases}
\ey
where $l_j=\text{sgn}(\beta_j)$ if $\beta_j\neq 0$ or $l_j$ takes value in $[-1,1]$ if $\beta_j= 0$.
 And $\partial l(\boldsymbol{\beta},\bxi)=\left\{\boldsymbol{\mu}=(\mu_1,\ldots,\mu_{p+D_n})'\in \R^{p+D_n}\right\}$ has the following expression,
 \by
 \mu_j=
 \begin{cases}H'_{\lambda}(\beta_j),&j=1,2,\ldots,p,\\
0,&j=p+1,\ldots,p+D_n;
 \end{cases}
 \ey

Consider the subgradient of $L_n(\boldsymbol{\beta},\bxi)$ at the oracle estimator $(\hat{\bbeta}^*,\hat{\bxi}^*)$. Actually,  Lemma \ref{Subgradient of oracle estimator} shows that under the Conditions 3.1-3.4 and the following so-called Beta-min condition,
\begin{condition}[Beta-min condition]\label{C(5)}
	There exist positive constants $C_4$ and $C_5$ such that $C_{3}< C_{4} < 1$ and
	\by
	n^{(1-C_4)/2}\underset{1\leq j \leq q_n}{\min}|\bbeta_{j}^*| \geq C_5,
	\ey
\end{condition}
for the oracle estimator $(\hat{\bbeta}^*,\hat{\bxi}^*)$,
\be
s_j(\hat{\bbeta}^*,\hat{\bxi}^*)&=&0,~~j=1,\ldots,q_n ~\text{or}~ j=p+1,\ldots,p+D_n,\\
|s_j(\hat{\bbeta}^*,\hat{\bxi}^*)|&\leq&\lambda,~~j=q_n+1,\ldots,p.
\ee

Define $\mathcal{E}(\lambda)$ be the set of local minima of $L(\boldsymbol{\beta},\bxi)$ with the tuning parameter $\lambda$. The following theorem builds up the relationship between the oracle estimator and the penalized nonconvex optimization problem (\ref{non-convex regularized problem}): with probability tending to one, the oracle estimator $(\hat{\bbeta}^*,\hat{\bxi}^*)$ is a local minimizer of $L(\boldsymbol{\beta},\bxi)$.
\begin{theorem} \label{ oracle property} Assume Conditions 3.1-3.5 are satisfied. If the tuning parameter  $\lambda = o\left(n^{-(1-  C_4)/2}\right)$, $q_n=o(n\lambda^2)$, $k_n = o(n\lambda^2)$ and $p = o\big((n\lambda^2)^k\big)$ , then we have that with probability tending to one,
	the oracle estimator $(\hat{\bbeta}^*,\hat{\bxi}^*)$ lies in the set $\mathcal{E}(\lambda)$ consisting of local minima of $L(\boldsymbol{\beta},\bxi)$, i.e.,	
\be
\P((\hat{\bbeta}^*,\hat{\bxi}^*)\in \mathcal{E}(\lambda))\rightarrow 1, \text{as}~n\to \infty.
\ee
\end{theorem}

	By the constraints on $\lambda$, we can see that $p=p(n)=o(n^{C_4k})$. So the moment condition and the signal strength directly influence the dimensionality our proposed method can handle. We should address that if the regression error is heavy-tailed and  has only finite polynomial  moments, $p$ can be at most a certain power of $n$.
	If $\epsilon_i$ has all the moments, this asymptotic result holds when $p=O(n^{\tau})$ for any $\tau>0$ since $\ep(\epsilon_i^{2k}|\x_i)<\infty$ for all $k>0$. What's more, if the error $\epsilon$ follows gaussian or sub-gaussian distribution, it can be shown that our method can be applied to ultra-high dimension .

\section{Simulation}
\label{simu}
\setcounter{definition}{0}\setcounter{definition}{0}
\setcounter{equation}{0}\setcounter{lemma}{0}
\setcounter{proposition}{0}\setcounter{theorem}{0}
\setcounter{remark}{0}\setcounter{corollary}{0}

In this section, we assess the finite sample performances of the proposed regularized expectile regression. For the choice of the general folded concave penalty function $P_{\lambda}(t)$, here we use the SCAD penalty as an example \footnote{The same procedure can be carried out with the MCP penalty or other penalties belonging to the general folded nonconvex penalty function}. For convenience, denote the penalized partially linear additive \textbf{E}xpectile regression with the SCAD penalty by E-SCAD for short.

We adopt a high-dimensional partially linear additive model from \cite{SaW16}. In this data generation procedure,
 firstly, the quasi-covariates $\tilde{\x}=(\tilde{x}_1,\ldots,\tilde{x}_{p+2})'$ is generated from the multivariate normal distribution $N_{p+2}(\mathbf{0},\Sigma)$ where $\Sigma=(\sigma_{ij})_{(p+2)\times (p+2)}$, $\sigma_{ij}=0.5^{|i-j|}$ for $i,j=1,\ldots,p+2$. Then we set $x_1=\sqrt{12}\Phi(\tilde{x}_1)$ where $\Phi(\cdot)$ is the cumulative distribution function of the standard normal distribution and $\sqrt{12}$ scales $x_1$ to have standard deviation 1. Furthermore, let $z_1=\Phi(\tilde{x}_{25})$ and $z_2=\Phi(\tilde{x}_{26})$, $x_i=\tilde{\x}_i$ for $i=2,\ldots,24$ and $x_i=\tilde{x}_{i+2}$ for $i=25,\ldots,p$.
  Then the response variable $y$ is generated from the following sparse model,
\bq
y=x_6\beta_6+x_{12}\beta_{12}+x_{15}\beta_{15}+x_{20}\beta_{20}+\sin(2\pi z_1)+z_2^3+\epsilon,
\eq
where $\beta_j = 1$ for $j = 6, 12, 15$ and $20$ and $\epsilon$ is independent of the covariates $\x$. To figure out how the proposed method performs when the error $\epsilon$ shares heavy-tailed distributions or not, we consider the following 2 scenarios,
\bn[(1)]
\item Standard normal distribution $N(0,1)$;

\item Standard t-distribution with degrees of freedom 5, $t_5$.
\en

As expected, when $\epsilon_{i}$ are an i.i.d. sequence, i.e., the homogeneous case, the simulation results (not summarized here for simplicity) show that our proposed method shares good performances in coefficient estimation, nonparametric approximation accuracy and model selection.

Now let us focus our attention to the heteroscedastic case. Here we assume $\epsilon=0.70 x_1\varsigma$ where $\varsigma$ is independent of $x_1$ and follows the two distributions above. In this situation, besides parameter estimation and model selection, we also want to test whether the proposed method can be used to detect heteroscedasticity.
From this data generation procedure above, we can see that the true coefficients in linear part are sparse and only includes 4 informative variables. $x_1$ should also be regarded as the significant variable since it plays an essential role in the conditional distribution of $y$ given the covariates and results in heteroscedasticity.

For comparison purpose, in this simulation, we also investigate the performance of the Lasso-type regularized expectile regression (E-Lasso for short). One may use $L_1$-penalty instead of SCAD penalty in penalized expectile regression and solve the following optimization problem,
\bq\label{lasso}
\underset{\boldsymbol{\beta} \in \R^{p},\bxi\in\R^{D_n}}{\arg\min}\frac{1}{n}\sum_{i=1}^n\phi_{\alpha}(y_i-\x_i'\boldsymbol{\beta}-\bPi(\z_i)'\bxi)+\lambda\sum_{j=1}^p|\beta_j|.
\eq
We aim to show the differences when we use $L_1$-penalty or the SCAD penalty in regularized expectile regression and furthermore tell the reason why we choose the folded concave penalty instead of  $L_1$-penalty  in this article. Besides, we introduce the oracle estimator (\ref{Oracle Estimator Representation}) as the benchmark of estimation accuracy.

 We set sample size $n=300$ and  $p=400$ or 600. For expectile weight level, by the results in~\cite{New87}, positions near the tail seem to be more effective for testing heteroscedasticity, so we consider two positions: $\alpha=0.10,0.90$. Note that when $\alpha=0.50$, expectile regression is exactly the classical ordinary least squares regression, so we also consider the position $\alpha=0.50$ so as to show our proposed method can be used to detect heteroscedasticity when $\alpha\neq0.50$. Given the expectile weight level $\alpha$, there are two tuning parameters in SCAD penalty function, $a$ and $\lambda$. We follow the suggestion proposed by \cite{Fan01} and set $a=3.7 $ to reduce the computation burden. For the tuning parameter $\lambda$, we generate another tuning data set with size $10n$ and choose the $\lambda$ that minimizes the prediction expectile loss error calculated on the tuning data set. For the nonparametric components, we adopt the cubic B-spline with 3 basis functions for each nonparametric function.

We repeat the simulation procedure 100 times and evaluate the performance in term of the following criteria:
\bi
\item AE:~the average absolute estimation error defined by $\sum_i^p|\hat{\beta}_j-\beta_j^*|$.
\item SE:~the average square estimation error defined by $\sqrt{\sum_i^p|\hat{\beta}_j-\beta_j^*|^2}$.
\item ADE:~the average of the average absolute deviation (ADE) of the fit of the nonlinear part defined by $\frac{1}{n}\sum_{i=1}^n|\hat{g}(\z_i)-g_0(\z_i)|$
\item Size:~the average number of nonzero regression coefficients $\hat{\beta}_j\neq0$ for $j=1,\ldots,p$.  Given the role of $x_1$, the true size of our data generation model supposes to be 5.
\item F:~the frequency that $x_6,x_{12},x_{15},x_{20}$ are selected during the 100 repetitions.
\item F1:~the frequency that $x_1$ is selected during the 100 repetitions.
\ei

\begin{table}[!htp]
	\centering
	\fontsize{8}{8}\selectfont
	\begin{threeparttable}
		\caption{Simulation results when $n=300$, $p=400$.}
		\label{tab:performance_comparison 1}
		\begin{tabular}{cccccccc}
			\toprule
			\multirow{2}{*}&\multirow{2}{*}{Criteria}& \multicolumn{3}{c}{$N(0,1)$}&\multicolumn{3}{c}{${t_5}$}\cr
			\cmidrule(lr){3-5} \cmidrule(lr){6-8}
			&&E-SCAD&E-Lasso&Oracle&E-SCAD&E-Lasso&Oracle\cr
			\midrule
			\multirow{5}{*}{$\alpha=0.10$}
			&AE&0.76(0.27)&1.94(0.42)&0.83(0.17)&1.21(0.77)&3.01(1.22)&1.13(0.33)\cr
			&SE&0.47(0.20)&0.59(0.11)&0.56(0.11)&0.62(0.32)&0.83(0.21)&0.72(0.18)\cr
			&ADE& 0.54(0.10)&0.67(0.10) &0.26(0.08)& 0.62(0.10)&0.73(0.15) &0.36(0.13)\cr
			&Size&6.79(1.73)&24.05(5.51)&-&8.46(3.63)&28.25(8.82)&-\cr
			&F,F1&100,~87&100,~97&-&100,~73&100,~94&-\cr
			\midrule
			\multirow{5}{*}{$\alpha=0.50$}
			&AE&0.31(0.14)&1.26(0.30)&0.28(0.11)& 0.47(0.16)&1.49(0.23)&0.41(0.13)\cr
			&SE&0.18(0.08)& 0.41(0.09)&0.16(0.06)&0.25(0.10)&0.50(0.09)&0.22(0.07)\cr
			&ADE&0.38(0.24) & 0.37(0.24)&0.18(0.05)& 0.44(0.18)&0.43(0.18)& 0.32(0.12)\cr
			&Size&4.64(0.78)&21.78(4.86)&-&6.49(1.42)&20.43(3.01)&-\cr
			&F,F1&100,~0&100,~8&-&100,~0&100,~8&-\cr
			
			\midrule
			\multirow{5}{*}{$\alpha=0.90$}
			&AE&0.74(0.27)&1.76(0.36)&0.82(0.16)&1.07(0.52)&2.94(1.23)&1.12(0.31)\cr
			&SE&0.47(0.20)&0.59(0.09)&0.56(0.11)&0.59(0.26) &0.84(0.19)&0.71(0.18)\cr
			&ADE&0.49(0.14)& 0.76(0.22)&0.25(0.09)& 0.52(0.37)& 0.68(0.13)& 0.38(0.13)\cr
			&Size&6.21(1.39)&19.54(4.91)&-&7.74(3.19)&26.06(9.19)&-\cr
			&F,F1&100,~88&100,~96&-&100,~76&100,~87&-\cr
			
			\bottomrule
		\end{tabular}
	\end{threeparttable}
\end{table}

\begin{table}[!htp]
	\centering
	\fontsize{8}{8}\selectfont
	\begin{threeparttable}
		\caption{Simulation results when $n=300$, $p=600$.}
		\label{tab:performance_comparison 2}
		\begin{tabular}{cccccccc}
			\toprule
			\multirow{2}{*}&\multirow{2}{*}{Criteria}&\multicolumn{3}{c}{$N(0,1)$}&\multicolumn{3}{c}{${t_5}$}\cr
			
			\cmidrule(lr){3-5} \cmidrule(lr){6-8}
			&&E-SCAD&E-Lasso&Oracle&E-SCAD&E-Lasso&Oracle\cr
			\midrule
			\multirow{5}{*}{$\alpha=0.10$}
			&AE&0.97(0.27)&2.11(0.44)&0.86(0.17)&1.36(0.85)&3.74(1.16)&1.14(0.25)\cr
			&SE&0.54(0.14)&0.64(0.10)&0.57(0.10)&0.66(0.29)&0.86(0.17)&0.72(0.15)\cr
			&ADE&0.50(0.27)&0.62(0.07)&0.24(0.07)& 0.81(0.51)&0.95(0.29)&0.38(0.12)\cr
			&Size&9.68(2.78)&26.55(5.60)&-&10.67(4.71)&42.53(9.50)&-\cr
			&F,F1&100,~96&100,~97&-&100,~77&100,~86&-\cr
			\midrule
			\multirow{5}{*}{$\alpha=0.50$}
			&AE&0.31(0.13)&1.50(0.36)&0.35(0.11)&0.63(0.27)&1.85(0.52)&0.43(0.13)\cr
			&SE&0.17(0.07)&0.43(0.09)&0.18(0.06)&0.32(0.12)&0.55(0.10)&0.23(0.07)\cr
			&ADE&0.38(0.23)&0.50(0.19)&0.18(0.04)&0.18(0.02)&0.20(0.06)&0.23(0.06)\cr
			&Size&5.61(1.55)&29.16(6.05)&-&7.36(3.16)&26.74(7.11)&-\cr
			&F,F1&100,~1&100,~10&-&100,~3&100,~5&-\cr
			
			\midrule
			\multirow{5}{*}{$\alpha=0.90$}
			&AE&0.82(0.27)&1.80(0.39)&0.83(0.15)&1.12(0.53)&3.18(1.64)&1.16(0.28)\cr         &SE&0.49(0.19)&0.64(0.10)&0.56(0.10)&0.60(0.28)&0.88(0.23)&0.72(0.15)\cr
			&ADE&0.40(0.21)&0.58(0.10)&0.24(0.09)&0.33(0.04)&1.09(0.54)&0.40(0.22)\cr
			&Size&7.72(2.14)&17.60(4.67)&-&8.32(3.36)&28.80(10.98)&-\cr
			&F,F1&100,~88&100,~94&-&99,~79&100,~83&-\cr
			\bottomrule
		\end{tabular}
	\end{threeparttable}
\end{table}

Table \ref{tab:performance_comparison 1} and Table \ref{tab:performance_comparison 2}  summarize the simulation results, corresponding to $p=400$ and $p=600$ respectively. In general, compared to E-Lasso, E-SCAD has a much better estimation accuracy and tends to select a much smaller model, and its performances are much closer to that of the oracle estimator. So we may tend to use the SCAD penalty  instead of the Lasso penalty in practise. As for further performances of E-SCAD,
first note that in our simulation setting  $m_{\alpha=0.5}(\epsilon|\x,\z)=0$,
then Theorem 3.2 tells that,  at this moment, E-SCAD has better performances in estimation accuracy and model selection than those with other weight levels $\alpha =0.1~\text{or}~0.9$, justified by
 the "AE","ADE" and "SE" results in Table \ref{tab:performance_comparison 1} and Table \ref{tab:performance_comparison 2}.  However,  the variance heterogeneity does not show up in $m_{\alpha=0.50}(y|\x,\z)$ so that at this situation E-SCAD can not pick $x_1$,  the active variable resulting in heteroscedasticity. Expectile regression with different weights can actually help solve this problem.
We can see that at $\alpha =0.1~\text{and}~0.9$,  $x_1$ can be identified  as the active variable with high frequency. On the other hand, from  Table \ref{tab:performance_comparison 1} to Table \ref{tab:performance_comparison 2}, as $p$ increases, we can see that the performances of E-SCAD get a little worse, and this change is more obvious in $t_5$ case. We have to say that the dimensionality our proposed method can handle is influenced by the heavy-tailed characteristics of the error. Our theoretical study indicates that if the regression error has only finite moments like in the $t_5$ case( $\ep\epsilon^{4+\delta}<\infty$a for $\delta \in (0,1)$), $p$ can be at most a certain power of $n$.

\section{Real Data Application}
\label{Real Data}
\setcounter{definition}{0}\setcounter{definition}{0}
\setcounter{equation}{0}\setcounter{lemma}{0}
\setcounter{proposition}{0}\setcounter{theorem}{0}
\setcounter{remark}{0}\setcounter{corollary}{0}

Low infant birth weight has always been a comprehensive quantitative trait, as it affects directly the post-neonatal mortality, infant and childhood morbidity, as well as its life-long body condition. Thus, on purpose of public health intervention, scientists have long put considerable investigation onto the low birth weight's determinants, see~\cite{Kra87}, who investigated 43 potential determinants and used a set of priori methodological standards to assess the existence and magnitude of the effect from potential factor to low birth weight. And~\cite{Tur12} used gene promoter-specific DNA methylation levels to identify genes correlated to low birth weight, with cord blood and placenta samples collected from each newborn. \cite{Voa11} collected samples of peripheral blood, placenta and cord blood from pregnant smokers$(n = 20)$ and gravidas$(n = 52)$ without significant exposure to cigarettes smoke. Their purpose was to identify the tobacco smoke-related defects, specifically, the transcriptome alterations of genes induced by the smoke. As the infant's birth weight was recorded along with the age of mother, gestational age, parity, maternal blood cotinine level and mother's BMI index, we consider using this data set to depict the infant birth weight's determinants. With a total of 65 observations contained in this genetic set, the gene expression profiles were assayed by Illumina Expression Beadchip v3 for the $24,526$ genes transcripts and then normalized in a quantile method.

To investigate the low birth weight of infant, we apply our partially linear additive penalized expectile regression model onto this data set. We consider to include the normalized genetic data, clinic variables parity, gestational age, maternal blood cotinine level and BMI as part of the linear covariates. And we take the age of mother as the nonparametric part to help explain nonlinear effect, according to~\cite{Voa11}. For sake of the possibly existing heteroscedasticity of these data and to dissect the cause of low infant birth weight, the analysis is carried out under three different expectile levels $\alpha = 0.1, 0.3$ and $0.5$. And in each scenario, feature screening methods could be used to select the top 200 relevant gene probes, see~\cite{Fan08} and \cite{He13}. Here in our data analysis, we choose to use the SCAD penalty (see Example 2.1) in our regularized framework and denote it by E-SCAD for short.  Other nonconvex  penalty functions like MCP penalty function could also be applied with our model, for which we don't give unnecessary details. For comparison purpose, we also consider penalized semiparametric regression with $L_1$ norm penalty, i.e., the Lasso-type regularized framework, named as E-Lasso. As recommended in~\cite{Fan01}, the parameter $a$ in the SCAD penalty is set to be 3.7 in order to reduce computation burden. As for the tuning parameter $\lambda$,  here we adopt the five-folded cross validation strategy to determine its value for both E-SCAD and E-Lasso.

First, we apply our proposed E-SCAD method to the whole data set at three different expectile levels $\alpha = 0.1, ~0.3$ and $0.5$. And for each level, the set of selected variables in the linear part of our model is denoted by $\hat{\mathcal{A}}_\alpha$, along with its cardinality $|\hat{\mathcal{A}}_\alpha|$. Taking the possibly existing heteroscedasticity into consideration, we also display the number of overlapped selected variables under different expectile levels, denoted by $|\hat{\mathcal{A}}_{0.1}\cap\hat{\mathcal{A}}_{0.5}|$ and  $|\hat{\mathcal{A}}_{0.3}\cap\hat{\mathcal{A}}_{0.5}|$. The number of selected variables and overlapped variables are reported in Table \ref{tab:Real1}. Next, we randomly partition the whole data set into a training data set of size 50 and a test set of size 15. Then, E-SCAD is used on the training set to obtain regression coefficients $\hat{\bbeta}$, and the estimated coefficients are used to predict the responses of 15 individuals in the test set. We repeat the random splitting process $100$ times. The variable selection results under random partition scenario are also shown in Table \ref{tab:Real1}. We also report the mean absolute error $L_1 = \frac{1}{24}{\sum_{i \in test\,set}}|y_i - x_i'\hat{\bbeta}|$, and the mean squared error, $L_2 = \frac{1}{24}\sqrt{{\sum_{i \in test\,set}}(y_i - x_i'\hat{\bbeta})^2} $ for prediction. Also, a boxplot of those two error values has been displayed in Figure 1.

\begin{table}[!htp]
 \centering
  \fontsize{5.6}{8}\selectfont
  \begin{threeparttable}
  \caption{Numeric results at three expectile levels}
  \label{tab:Real1}
    \begin{tabular}{cccccccc}
    \toprule
   \multirow{2}{*}&  \multirow{2}{*}{\textbf{Criteria}} & \multicolumn{2}{c}{$\alpha=0.1$} & \multicolumn{2}{c}{$\alpha=0.3$} & \multicolumn{2}{c}{$\alpha=0.5$} \cr
    \cmidrule(lr){3-4}\cmidrule(lr){5-6}\cmidrule(lr){7-8}
                                                        &&E-SCAD&E-LASSO&E-SCAD&E-LASSO&E-SCAD&E-LASSO\cr
    \midrule
   \multirow{4}{*}{All Data}
   & $L_1$ & 0.66   & 0.67& 0.60&0.53 &0.38 &0.34 \\
    &$L_2$ & 0.12& 0.11&0.10 & 0.09&0.06 &0.05\\
    &$\hat{\mathcal{A}} _\alpha$ & 7.00  &    8.00& 9.00 &19.00 &14.00 &20.00 \\
    &$\hat{\mathcal{A}}_\alpha \bigcap\hat{\mathcal{A}}_{0.5}  $ & 1&1 &3 &3 & -&- \\
    \midrule
      \multirow{4}{*}{Random Partition}
       & ${L_1}$ &     0.90(0.21)  &  0.74(0.17) & 0.81(0.17)& 0.59(0.13)&  0.89(0.20)&  0.41(0.08)\\
    &${L_2}$ &         0.30(0.07) & 0.26(0.06)&0.27(0.06) & 0.19(0.04)& 0.30(0.06)&0.13(0.02) \\
    &$\hat{\mathcal{A}} _\alpha$ & 5.72(1.91)   &8.19(2.74) & 9.00(2.72)& 13.94(3.03)&   4.72(1.83)& 20.25(2.67)\\
    &$\hat{\mathcal{A}}_\alpha \bigcap\hat{\mathcal{A}}_{0.5}$ & 3.86 (1.6433) &  1.16(0.39) &2.27(1.13) & 3.25(1.18) & -&- \\
    \bottomrule

      \end{tabular}
    \end{threeparttable}
\end{table}
\begin{figure}[H]
\label{fig:real1}
\centering
\includegraphics[width=5in,height = 3in]{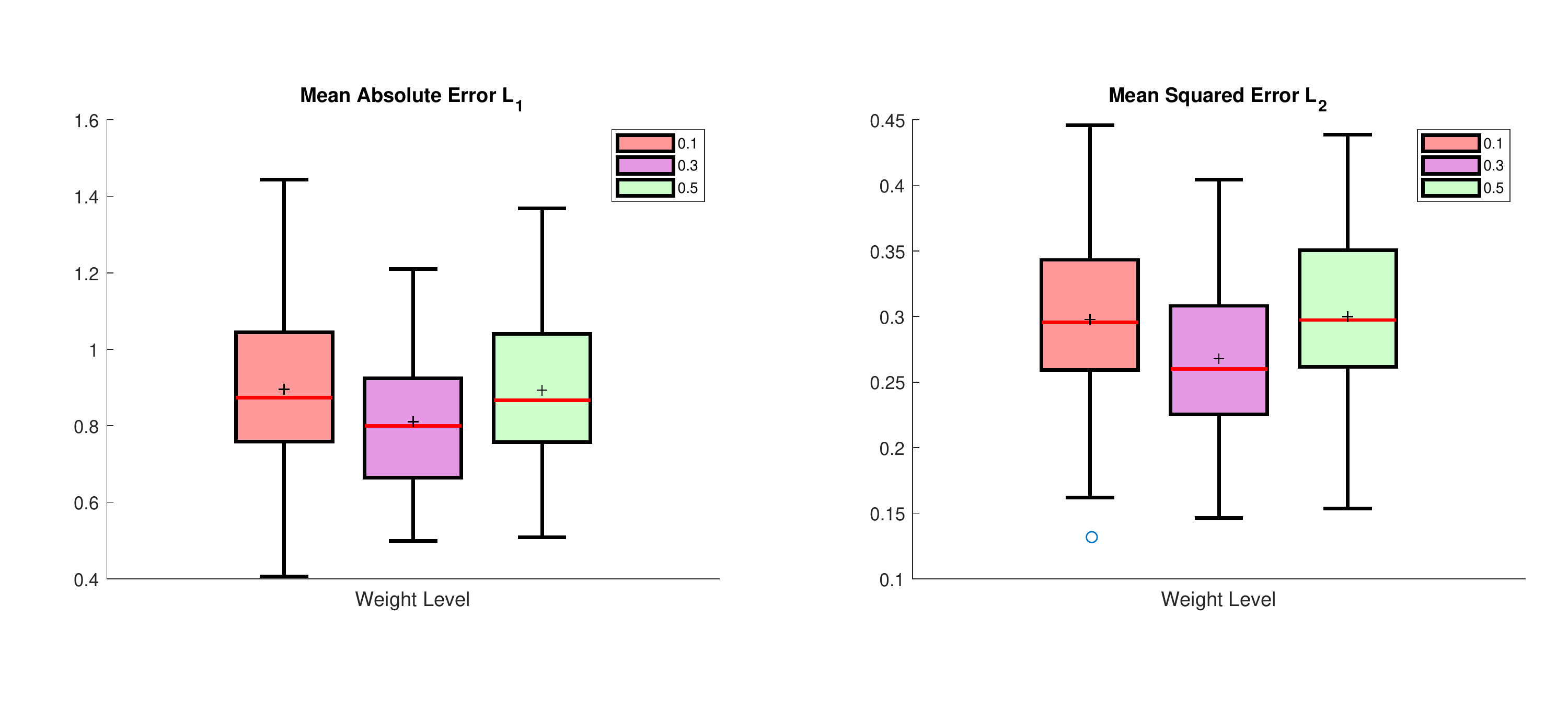}
\caption{Boxplots of Prediction Errors.}
\end{figure}

\begin{table}[!htp]
 \centering
  \fontsize{5.6}{8}\selectfont
  \begin{threeparttable}
  \caption{Top 6 Frequent Covariates Selected at Three Expectile Weight Levels  among 100 Partitions}
  \label{tab:Real3}
    \begin{tabular}{cccccc}
    \toprule
    \multicolumn{2}{c}{E-SCAD $\alpha = 0.1$} & \multicolumn{2}{c}{E-SCAD $\alpha=0.3$} & \multicolumn{2}{c}{E-SCAD $\alpha=0.5$} \\
     \cmidrule(lr){1-2}\cmidrule(lr){3-4}\cmidrule(lr){5-6}
     {Variables} &{Frequency}&{Variables} &{Frequency}&{Variables} &{Frequency}\\
     \midrule
PTPN3 &	34	&GPR50	&46		&PTPN3&	33\\
FXR1&	 40	&FXR1&	49	&	GPR50&	40\\
GPR50	 &43	&EPHA3	&50		&FXR1	&41\\
LEO1&	43	&LEO1	&59		&LEO1	&44\\
SLCO1A2	&63	&LOC388886	&65		&SLCO1A2	&65\\
	Gestational age&79&	Gestational age&	97&	Gestational age&		83\\

    \bottomrule
      \end{tabular}

    \end{threeparttable}
\end{table}
 As shown in Figure 1, the underlying models selected under three expectile  levels $\alpha = 0.1,0.3,0.5$ all lead to a relatively small prediction error. In Table \ref{tab:Real1}, the selected genes and corresponding cardinalities are different for different weight levels, which means that different levels of birth weight are influenced by different genes,
 an indication of heterogeneity in the data.  Table \ref{tab:Real3} tells us more about this fact. Gestational age is the most frequently selected covariate under all three scenarios, explaining the known fact that premature birth is usually accompanied by low birth weight. Besides, SLCO1A2, LEO1, FXR1 and GPR50 appear frequently in all three cases. Furthermore,  an interesting observation arises that the scenarios $\alpha =0.1$ and $\alpha =0.5$ perform similarily while the scenario $\alpha =0.3$ displays some different characteristics. Gene SLCO1A2 is selected with higher frequencies when $\alpha = 0.1$ and $\alpha = 0.5$. This gene is known to have resistance against drug use and moreover,  according to \cite{Voa11}, exposure to toxic compounds contained in tobacco smoke can be strongly associated with low birth weight. Gene EPHA3 is more frequently selected under the scenario $\alpha =0.3$ compared with other two cases. As shown in \cite{Lv18}, EPHA3 is likely to contribute tumor growth in human cancer cells, which may make pregnant women more sensitive to chemical compounds contained in cigarette smoke. And the study of \cite{Ku05} concludes that EPHA3's expression at both the mRNA and protein level is a critical issue during mammalian newborn forebrain's development. These results can well account for our analysis under different expectile values and may furthermore, arise more our attention upon the $\alpha = 0.3$ expectile value case, due to its specially selected results and potentially underlying biomedical meaning.

\section{Appendix}
\label{Appe}
\setcounter{definition}{0}\setcounter{definition}{0}
\setcounter{equation}{0}\setcounter{lemma}{0}
\setcounter{proposition}{0}\setcounter{theorem}{0}
\setcounter{remark}{0}\setcounter{corollary}{0}

\subsection{Explanation Materials before the proof.}

We are going to use the theoretically centered B-spline basis functions to simplify the proof. Notify B-spline basis functions $b_j(\cdot)$ in Section 2 can be centered as $B_j(z_{ik}) = b_{j+1}(z_{ik}) - \frac{\ep[b_{j+1}(z_{ik})]}{\ep[b_{1}(z_{ik})]}b_1(z_{ik})$, for $j = 1,\ldots,k_n+l$, satisfying $\ep[B_j(z_{ik})] = 0$. We denote a vector of centered basis functions as  ${\mathbf{w}} (z_{ik}) = (B_1(z_{ik}),\ldots,B_{k_n+l}(z_{ik}))^{'}$, the $J_n$-dimensional vector $\W(\z_i) = (k_n^{-1/2},{\mathbf{w}} (z_{i1})^{'},\dots{\mathbf{w}} (z_{id})^{'}))^{'}$, where $J_n = d(k_n +l) +1$ and the $n\times J_n$ matrix $W = (\W(\z_1),\ldots,\W(\z_n))^{'}$. According to \cite{Sch07}, there exists a vector $\gamma_0 \in \R^{J_n}J$  such that $\sup_{\z\in[0,1]^d}|g_0(\z) - \W(\z)^{'}\gamma_0| = O(k_n^{-r})$. Define a new pair of minimizer $(\hat{\c}_A,\hat{\bgamma})$ as
\by
(\hat{\c}^*_A,\hat{\bgamma}) = \underset{(\c_A,\bgamma)}{\arg \min}\frac{1}{n}\sumn\phi_\alpha(y_i - \x_{Ai}^{'}\c_A-\W(\z_i)^{'}\bgamma)
\ey
Same to  Section 3, we write ${\bgamma} = (\gamma_0,\bgamma^{'}_1,\ldots,\bgamma^{'}_d)^{'}$ with $\gamma_0 \in \R$ and $\bgamma_j  \in \R^{k_n+l}$, for $j = 1,\ldots,d$. We define the estimators for those nonparametric functions in this new way as $\tilde{g}_j(\z_i) = \w(z_{ij})^{'}\hat{\bgamma_j},~ j = 1,\ldots,d$ and the estimator for $\mu_{0}$ is $\tilde{\mu} = k_n^{-1/2}\hat{\gamma}_0$. Thus, the estimator for $g_0(\z_i)$ is $\tilde{g}(\z_i) = \W(\z_i)^{'}\hat{\bgamma} = \tilde{\mu} + \sum_{j =1}^{d}\tilde{g}_j(z_{ij})$.

It should be notified that $\hat{\c}^*_A = \hat{\bbeta}^*_A$, that is the centered B-spline basis functions do not alter the linear part, according to \cite{SaW16}. Also, it has been shown that the original estimator of nonparametric functions can be derived from the  new ones as $
\hat{\mu} = \tilde{\mu} + \frac{1}{n}\sumn\sum_{j = 1}^{d}\tilde{g}_j(z_{ij})$ and $
\hat{g}_j(z_{ij}) = \tilde{g}_j(z_{ij}) - \frac{1}{n}\sumn{\tilde{g}_j(z_{ij})}$. Thus, we have $\hat{g}(\z_i) = \tilde{g}(\z_i)$.

\subsection{Notation.}

Throughout the proofs, we denote C a positive constant which does not depend on n and may vary from line to line. For a vector $\btheta$, $||\btheta||$ refers to its $L_2$ norm. For a matrix X, $||X|| = \sqrt{\lambda_{max}(X^{'}X)}$ denotes its spectral norm. Also, we introduce the following notations.

Furthermore, we have following notations throughout the appendix.
\by
B_n &=& \text{diag}(w_1,\ldots,w_n) \in \R^{n\times n},\\
W_n &= & n^{-1}\sum_i w_i\delta_i\delta_i^{'} \in \R^{n\times n},\\
P&=&W(W^{'}B_nW)^{-1}W^{'}B_n\in \R^{n\times n},\\
W_B^2 &= &W^{'}B_nW \in \R^{J_n\times J_n},\\
X^{*}&=&(\x_1^{*},\ldots,\x_n^{*})^{'} = (I_n - P)X_A\in \R^{n\times q_n},\\
\tilde{\x}_i &= &n^{-1/2}\x^*_i \in \R^{q_n},\\
\tilde{\W}(\z_i) &= &W_B^{-1}\W(\z_i)\in \R^{J_n},\\
\tilde{s}_i &=& (\tilde{\x}_i ^{'},\tilde{\W}(\z_i) ^{'})^{'} \in \R^{J_n+q_n},\\
\btheta_1 &=& \sqrt{n}(\c_A-\beta_{A}^*)\in \R^{q_n} ,\\
\btheta_2&=& W_B(\bgamma-\bgamma_0)+W_B^{-1}W^{'}B_nX_A(\c_A - \bbeta_{A}^*)\in \R^{q_n} ,\\
u_{ni} &=& \W(\z_i)^{'}\bgamma_0 -g_0(\z_i)
\ey

 Under the new notation system, the objective loss function can be displayed as
 \by
 \frac{1}{n}\sumn\phi_\alpha(y_i - \x_{Ai}^{'}\c_A-\W(\z_i)^{'}\bgamma) = \frac{1}{n}\sumn\phi_\alpha(\epsilon_i - u_{ni} - \tilde{\x}^{'}_i\btheta_1- \tilde{\W}(\z_i)^{'}\btheta_2),
 \ey
and the transformed minimizers are
\by
(\hat{\btheta}_1,\hat{\btheta}_2) = \underset{(\btheta_1,\btheta_2)}{\arg\min}\frac{1}{n}\sumn\phi_\alpha(\epsilon_i - u_{ni} - \tilde{\x}^{'}_i\btheta_1- \tilde{\W}(\z_i)^{'}\btheta_2)
\ey

\subsection{Lemmas and Theoretical Proof of Theorems.}

First, we are going to show some properties for the B-spline basis vectors and design matrices.
\begin{lemma}[Properties of Loss Function] \label{Prop of Loss}
	The loss function $\phi_\alpha(\cdot)$ is defined in (\ref{phi}) and we set $c_1\triangleq\min\{\alpha,1-\alpha\}$ and $c_2 \triangleq \max\{\alpha,1-\alpha\}$. Thus, $\phi_\alpha(\cdot)$ satisfies:
  \bn[(1)]
  \item  $ \phi_\alpha(u)$ is continuous differentiable. Moreover, for any $r,r_0\in \R$, we have
	\bq
	c_1\cdot(r-r_0)^2\leq \phi_{\alpha}(r)-\phi_{\alpha}(r_0)-\psi_{\alpha}(r_0)\cdot(r-r_0)\leq c_2\cdot(r-r_0)^2.
	\eq
  \item $ \psi_\alpha(u)$ is   Lipschitz continuous, which means for any $r,r_0\in \R$, we have
	\bq
2c_1|r-r_0|\leq |\psi_\alpha(r)-\psi_\alpha(r_0)| \leq 2c_2|r-r_0|.
	\eq
\item $ \varphi_\alpha(u)\leq 2c_2$ and $w_i \triangleq\ep[\varphi_\alpha(\epsilon_i)]$ is uniformly bounded away from zero.
\en
\end{lemma}
\proof Details for (1) and (2) can be found in \cite{GaZ16}. From the definition of $\varphi_\alpha(u)$, the third result is obvious.\qed

\begin{lemma}[Properties of Spline Basis Functions]\label{Prop of Spline}
Here are some properties for centered basis function vectors.
\bn[(1)]
\item $\max_{i} \ep||\W(\z_i)||\leq m_1$, for some positive constant $m_1$ for a sufficiently large n;
\item There exists positive constants $m_2$ and $m_2^{'}$ such that for n sufficiently large $m_2k_n^{-1}\leq\ep[\lambda_{min}(\W(\z_i)\W(\z_i)^{'})] \leq\ep[\lambda_{max}(\W(\z_i)\W(\z_i)^{'})]\leq m_2^{'}k_n^{-1} $;
\item There exists positive constant $m_3$ such that for n sufficiently large $\ep||W_B^{-1}||  \leq  m_3\sqrt{k_nn^{-1}}$;
 \item $\max_{i} ||\tilde{\W}(\z_i)|| = O_p(\sqrt{k_n/n})$
\en
\end{lemma}
\begin{proof}
The proof for (1) and (2) can be found in \cite{SaW16}. As for (3), we only need to show that $\ep[\lambda_{min}(W_B^2)] > Cnk_n^{-1}$, because $||W_B^{-1}|| = \lambda_{max}(W_B^{-1}) =\lambda_{min}^{-1/2}(W_B^{2})$ . Through the definition of $W_B$, we have
\by
\lambda_{min}(W_B^2) &=& \lambda_{min}(W^{'}B_nW) \\
&=& \lambda_{min}(\sumn \ep[\varphi_\alpha(\epsilon_i)]\W(\z_i)\W(\z_i)^{'})\\
&\geq& C\sumn  \lambda_{min}(\W(\z_i)\W(\z_i)^{'})\\
&\geq&C k_n^{-1}n,
\ey
where the last inequality follows from the second result in this lemma.

As for (4), we have
\by
||\tilde{\W}(\z_i)||^2 = \W(\z_i)^{'}W_B^{-2}\W(\z_i)  \leq ||\W(\z_i)||^2 \cdot \lambda_{max}^2(W_B^{-1}) = O_p(\frac{k_n}{n}).
\ey
And this finishes the proof.\qed
\end{proof}

\begin{lemma}[Properties of Design Matrix]\label{Prop of Design}
We show some properties for after-projection new design matrix $X^*$:
\bn[(1)]
\item $\lambda_{max} (n^{-1}X^{*'}X^*) \leq C$ with probability one, where C is a positive constant;
\item $n^{-1/2}X^* = n^{-1/2}\Delta_n^* + o_p(1)$. Also, $n^{-1}X^{*'}B_nX^* = W_n + o_p(1)$;
\item $\sumn w_i\tilde{\x}_i\tilde{\W}(\z_i)^{'} = \mathbf{0}.$
\en
\end{lemma}
\begin{proof}
  First, we show that $n^{-1/2}||H-PX_A||=o_p(1)$. Notice that $X^*$ is a weighted  projection from design matrix $X_A$ to the space conducted by spline basis matrix $W$'s column vectors. In other words, we set $\nu_j \in \R^{J_n}$ as $\nu_j = \underset{\nu \in \R^{J_n}}{\arg\min} \sumn \ep[\varphi_\alpha(\epsilon_i)](X_{Aij} - \W(\z_i)^{'}\nu)^2$. Thus, according to the result of weighted least square regression, we have $\{\tilde{h}_j(\z_i)\}_{n\times q_n} ={W(W^{'}B_nW)^{-1}W^{'}B_nX_A}$, where $ \tilde{h}_j(\z_i) = \W(\z_i)^{'}\nu_j$.

Then, by the definition of $X^*$, we have
\by\frac{1}{\sqrt{n}} X^* =   \frac{1}{\sqrt{n}} (I-P)X_A = \frac{1}{\sqrt{n}}\Delta_n + \frac{1}{\sqrt{n}}(H-PX_A).
\ey
and
\by
\frac{1}{n}\lambda_{max}\big((H-PX_A)^{'}(H-PX_A)\big)
 &\leq& \frac{1}{n}trace\big((H-PX_A)^{'}(H-PX_A)\big)\\
 &=&  \frac{1}{n}\sumn \sum^{q_n}_{j =1} (\tilde{h}_j(\z_i)-\hat{h}_j(\z_i))^2\\
& =&O_p(q_nn^{-2r/(2r+1)}) = o_p(1),
\ey
where the second last equality follows from \cite{Sto85}. Then the following proof is similar to Lemma 3 in \cite{SaW16}.\qed
\end{proof}

\begin{lemma}[Bernstein Inequality]\label{Bernstein inequality}
	Let $\xi_1,\xi_2,\ldots,\xi_n$ be independent mean-zero random variables, with uniform bounds $|\xi_i| \leq M $ and $v \geq \var(\sumn \xi_i)$. Then for every positive t,
    \by
    \P\left(|\sum_{i=1}^{n}\xi_i|>t \right) \leq 2\exp\left(-\frac{t^2}{2(v + Mt/3) }\right).
    \ey
	
\end{lemma}
\proof Details are in Lemma 2.2.9 of \cite{vaW96}.\qed

\begin{lemma} \label{ShL4}
If Conditions 3.1-3.4  hold, then we have
\by
\frac{1}{n}\sumn(\tilde{g}(z_i)-g_0(z_i))^2=o_p(\frac{d_n}{n}),
\ey
where $d_n=q_n+k_n$.
\end{lemma}
\begin{proof}
Define
\by
\Phi_i(a_n) &\triangleq& \Phi_i(a_n\theta_1,a_n\theta_2)= \phi_\alpha(\epsilon_i - a_n\tilde{\x}^{'}_i\theta_1-a_n\tilde{W}(z_i)^{'}\theta_2-u_{ni})
\ey
where $a_n$ refers to a sequence of positive numbers. And we notify
\by\ep_s[\cdot] &\triangleq& \ep[\cdot|x_i,z_i].\ey
Here we first show that $\forall \eta > 0$, there exists an $L>0$ such that
\bq\label{EQL4}
\P\Big(\inf_{||\theta||=L}\frac{1}{d_n}\sumn\big( \Phi_i(\sqrt{d_n}) - \Phi_i(0) \big)> 0\Big) > 1-\eta,
\eq
which implies with probability at least $1-\eta$ that there exists a local minimizer in the ball $\{ \sqrt{d_n}\theta:||\theta||\leq L\}$, that means the local minimizer $\hat{\theta}$ satisfies $||\hat{\theta}|| = O_p(\sqrt{d_n}).$

Set $\theta = (\theta^{'}_1,\theta^{'}_2)^{'}$. By Taylor Expansion, we have
\by
\Phi_i(a_n) &= &\phi_\alpha(\epsilon_i - a_n\tilde{s}^{'}_i\theta-u_{ni})\\
&=& \phi_\alpha(\epsilon_i - u_{ni}) - \psi_\alpha(\epsilon_i - u_{ni}) a_n\tilde{s}^{'}_i\theta + r_{i}(a_n)
\ey
and there exists $0<\xi_i<1$ s.t. \by r_{i}(a_n) =\frac{1}{2} \varphi_\alpha(\epsilon_i - u_{ni}-\xi_i a_n\tilde{s}^{'}_i\theta) (a_n\tilde{s}^{'}_i\theta)^2.\ey
Notice that apply Lemma \ref{Prop of Loss}, then
\by
|\psi_\alpha(\epsilon_i - u_{ni})-\psi_\alpha(\epsilon_i )| \leq 2c_2|u_{ni}| = O_p(k_n^{-r}) =o_p(1).
\ey

Thus, we have
\by
&&\frac{1}{d_n}\sumn \Phi_i(\sqrt{d_n}) - \Phi_i(0)\\
 &=&\frac{1}{d_n}\sumn \ep_s[\Phi_i(\sqrt{d_n}) - \Phi_i(0)]
 +\frac{1}{d_n}\sumn \Big(\Phi_i(\sqrt{d_n}) - \Phi_i(0) - \ep_s[\Phi_i(\sqrt{d_n}) - \Phi_i(0)]\Big)\\
 &=&\frac{1}{d_n}\sumn \ep_s[\Phi_i(\sqrt{d_n}) - \Phi_i(0)]\\
   &&+  \frac{1}{d_n}\sumn \Big(- \psi_\alpha(\epsilon_i )(1+o_p(1)) \sqrt{d_n}\tilde{s}^{'}_i\theta
+ \ep_s[ - \psi_\alpha(\epsilon_i )(1+o_p(1)) \sqrt{d_n}\tilde{s}^{'}_i\theta] \Big)\\
&&+   \frac{1}{d_n}\sumn \Big(r_{i}(\sqrt{d_n}) - \ep_s[ r_{i}(\sqrt{d_n})] \Big)\\
 &= &\frac{1}{d_n}\sumn \ep_s[\Phi_i(\sqrt{d_n}) - \Phi_i(0)]
- \frac{1}{\sqrt{d_n}}\sumn \psi_\alpha(\epsilon_i )(1+o_p(1)) \tilde{s}^{'}_i\theta\\
&&+\frac{1}{d_n}\sumn r_{i}(\sqrt{d_n}) - \ep_s[ r_{i}(\sqrt{d_n})] \\
&=& \Sigma_1 +\Sigma_2 +\Sigma_3
\ey
For $\Sigma_3$, define $D_i(a_n) = r_{i}(a_n) - \ep_s[ r_{i}(a_n)]$, we are going to show
\by
\underset{||\theta||\leq L}{\sup} ~d_n^{-1}\sumn |D_i(\sqrt{d_n})| = o_p(1).
\ey
Notify $F_{n1}$ as the event $\{\max_i||\tilde{s}_i|| \leq \alpha_1\sqrt{d_n/n }\}$. By {Lemma \ref{Prop of Spline}} and the fact $\max_i||\tilde{\x}_i||  = O(\sqrt{q_n/n })$, it implies that $\P(F_{n1})\rightarrow 1$. And for some positive constant $\alpha_2$ set $F_{n2}$ as the event $\{\max_i|u_{ni}|\leq\alpha_2k_n^{-r} \}$, then $\P(F_{n2})\rightarrow 1$ according to \cite{Sch07}.

Simplify $D_i(\sqrt{d_n})$ as $D_i$. Notice that $D_i$ are independent random variables satisfying $\ep D_i = 0$ and \by\max_i|D_i|\I(F_1\bigcap F_2) \leq \max_i C d_n|\tilde{s}^{'}_i\theta|^2 \I(F_1\bigcap F_2)\leq CL\frac{d_n^2}{n}.\ey
Thus,
\by
\var(D_i \I(F_{n1}\bigcap F_{n2})|x_i,z_i) &\leq& \ep_s[r_{i}^2(\sqrt{d_n})\I(F_{n1}\bigcap F_{n2})]\\
&\leq& \ep_s[\frac{1}{2} \varphi_\alpha(\epsilon_i - u_{ni}-\xi_i a_n\tilde{s}^{'}_i\theta) (a_n\tilde{s}^{'}_i\theta)^2]^2\\
& \leq& C d_n^2||\tilde{s}_i^{'}||^4||\theta||^4 \leq Cd_n^4/n^2.
\ey

We have
\by
\sumn \var(D_i \I(F_{n1}\bigcap F_{n2})|x_i,z_i) \leq Cd_n^4/n.
\ey
Applying Bernstein Inequality in Lemma \ref{Bernstein inequality}, for any positive constant $\epsilon$,
\by
\P(|\sumn D_i|> d_n\epsilon, F_{n1}\bigcap F_{n2}|x_i,z_i)
&\leq& 2\exp(\frac{-d_n^2\epsilon^2}{2(C\frac{d_n^4}{n}+C\frac{d_n^3L\epsilon}{3n})})\\
&\leq& C\exp(-\frac{n}{d_n^2})\rightarrow 0,
\ey
which implies that $\underset{||\theta||\leq L}{\sup} ~d_n^{-1}\sumn |D_i(\sqrt{d_n})| = o_p(1).$

Consider $\Sigma_1$, we can show that
$ |\varphi_\alpha(\epsilon_i - u_{ni}) - \varphi_\alpha(\epsilon_i )|=o_p(1)$.
It is sufficient to show that for any $\epsilon > 0$, $\P(|\varphi_\alpha(\epsilon_i - u_{ni}) - \varphi_\alpha(\epsilon_i )  | >\epsilon) \rightarrow 0$.  Notice that
\by
\P(|\varphi_\alpha(\epsilon_i - u_{ni}) - \varphi_\alpha(\epsilon_i )  | >\epsilon)
  &=& \max \{ \P(u_{ni}<\epsilon_i<0),\P(0<\epsilon_i<u_{ni}) \}\\
&\leq& C\cdot |u_{ni}| = O_p(k_n^{-r})\to 0
\ey
which implies
$|\varphi_\alpha(\epsilon_i - u_{ni}) - \varphi_\alpha(\epsilon_i )  | = o_p(1)$.

Then, through Taylor Expansion, it follows that
\by
&&\frac{1}{d_n}\sumn\ep_s\big[\Phi_i(\sqrt{d_n})- \Phi_i(0))\big]\\
&=&\frac{1}{d_n}\sumn\ep_s\big[\phi_\alpha(\epsilon_i - u_{ni}-\sqrt{d_n}\tilde{s}^{'}_i\theta) - \phi_\alpha(\epsilon_i - u_{ni})\big]\\
&=&\frac{1}{d_n}\sumn\ep_s\Big[- \psi_\alpha(\epsilon_i - u_{ni}) \sqrt{d_n}\tilde{s}^{'}_i\theta+\frac{1}{2} \varphi_\alpha(\epsilon_i - u_{ni}) {d_n}(\tilde{s}^{'}_i\theta)^2(1+o(1))\Big]\\
&=& \frac{1}{d_n}\sumn\ep_s\Big[- \psi_\alpha(\epsilon_i ) (1+o_p(1)) \sqrt{d_n}\tilde{s}^{'}_i\theta+\frac{1}{2} \varphi_\alpha(\epsilon_i ) (1+o_p(1)) {d_n}(\tilde{s}^{'}_i\theta)^2\Big]\\
&=& \sumn\ep_s\Big[\frac{1}{2} \varphi_\alpha(\epsilon_i ) (1+o_p(1)) (\tilde{s}^{'}_i\theta)^2\Big]\\
&=& \sumn\ep_s\Big[\frac{1}{2} \varphi_\alpha(\epsilon_i ) (\tilde{\x}^{'}_i\theta_1)^2(1+o_p(1)) \Big]+ \sumn\ep_s\Big[\frac{1}{2} \varphi_\alpha(\epsilon_i ) (\tilde{W}^{'}_i\theta_2)^2(1+o_p(1)) \Big] \\ &&+\sumn\ep_s\Big[\frac{1}{2} \varphi_\alpha(\epsilon_i ) \theta_1^{'}\tilde{\x}_i\tilde{W}^{'}_i\theta_2(1+o_p(1)) \Big] \\
&=& C\cdot\theta_1^{'}W_n\theta_1(1+o_p(1)) +C\cdot||\theta_2||^2(1+o_p(1)) \\
&=& O_p(||\theta||^2),
\ey
where $C$ is some positive constant and the last second equation applies Lemma \ref{Prop of Spline} and Lemma \ref{Prop of Design}.

As for $\Sigma_2$, notice $\ep_s[\psi_\alpha(\epsilon_i)\tilde{s}_i\theta] = 0$ and  $\ep[\psi_\alpha(\epsilon_i)]^2 =\ep\big[2\epsilon_i|\alpha-\I(\epsilon_i<0)|\big]^2\leq C\cdot \ep[\epsilon_i]^2\leq C $. Thus,
\by
\var(\psi_\alpha(\epsilon_i)\tilde{s}_i\theta|x_i,z_i)& = & \ep[\psi_\alpha(\epsilon_i)]^2(\tilde{s}_i\theta)^2\\
&\leq&C\frac{d_n}{n}||\theta||^2
\ey
which means $\var(\Sigma_2|x_i,z_i)\leq \frac{1}{d_n}\sumn\var(\psi_\alpha(\epsilon_i)\tilde{s}_i\theta|x_i,z_i) = O(||\theta||^2).$ Thus, $|\Sigma_2| = O_p(||\theta||)$. Hence, $\Sigma_2$ is dominated by $\Sigma_1$  for  sufficiently large $L$, thus { (\ref{EQL4})} holds.

Then from $||\hat{\theta}|| = O_p(\sqrt{d_n})$ and notice the definition of $\theta$, we have $ ||W_B(\hat{\gamma}-\gamma_0)|| = O_p(\sqrt{d_n})$. Notice that $\ep[\varphi_\alpha(\epsilon_i)]$ is uniformly bounded away from zero,then
\by
&&\frac{1}{n}\sumn\ep[\varphi_\alpha(\epsilon_i)](\tilde{g}(z_i)-g_0(z_i))^2\\
&=&\frac{1}{n}\sumn\ep[\varphi_\alpha(\epsilon_i)](W(z_i)^{'}(\hat{\gamma}-\gamma_0)-u_{ni})^2\\
&\leq&\frac{1}{n}C(\hat{\gamma}-\gamma_0)^{'}W_B^2(\hat{\gamma}-\gamma_0) + O_p(k_n^{-2r})
=O_p(\frac{d_n}{n}).
\ey
\qed
\end{proof}

\begin{lemma}
\label{ShL5}
Assume Conditions 3.1-3.4 hold. Set $\tilde{\theta}_1 = {n}^{-1/2}{(X^{*'}B_nX^*)}^{-1}X^{*'}\psi_\alpha(\epsilon)$, where $\psi_\alpha(\epsilon) = (\psi_\alpha(\epsilon_1),\ldots,\psi_\alpha(\epsilon_n))^{'} = (\psi_\alpha(\epsilon_1),\ldots,\psi_\alpha(\epsilon_n))^{'}$, then
\by
||\tilde{\theta}_1|| = O_p(\sqrt{q_n}).
\ey
\end{lemma}
\begin{proof}
From the definition of $\tilde{\theta}_1$ and {Lemma \ref{Prop of Design}}, we have
 \by\tilde{\theta}_1 &=& \frac{1}{\sqrt{n}}(W_n+o_p(1))^{-1}(\Delta_n^{'}+o_p(1))\psi_\alpha(\epsilon)\\
 &=& \frac{1}{\sqrt{n}}W_n^{-1}(\Delta_n^{'}\psi_\alpha(\epsilon)(1+o_p(1))\\
 & =& \frac{1}{\sqrt{n}}\sumn W_n^{-1}\delta_i\psi_\alpha(\epsilon_i)\\
 &\triangleq& \sumn D_{n,i},\ey
  And $D_{n,i}$ are independent random variables satisfying $\ep\tilde{\theta}_1= 0$ and
 \by
 \ep||\tilde{\theta}_1||^2 &=& \sumn\ep||D_{n,i}||^2\\
 &=&\frac{1}{n}\sumn\ep[\psi_\alpha^2(\epsilon_i)(\delta_i^{'}W_n^{-2}\delta_i)]\\
 &\leq&\frac{C}{n}\sumn \ep||\delta_i||^2\\
 &=& O_p(q_n)
 \ey
 where the second last inequality follows from the fact
 \by
 \ep\psi_\alpha^2(\epsilon_i)&=&\ep[2|\alpha-\I(\epsilon_i<0)|\epsilon_i]^2\\
 &\leq& C\ep\epsilon^2_i=O(1)
 \ey
 Thus, we have
 \by
||\tilde{\theta}_1|| = O_p(\sqrt{q_n}).
\ey\qed
\end{proof}

\begin{lemma}\label{L3.21}
Assume Condition 3.1-3.4 hold, then for any positive constant C,
\by
\underset{ ||\theta_2||\leq C\sqrt{d_n}}{\sup}\frac{1}{n}\sumn | \I (\epsilon_i < u_{ni} + \tilde{\W}(\z_i)^{'}\theta_2)-\I (\epsilon_i < 0) | = o_p(1).
\ey
\end{lemma}
\begin{proof}
Partition the left side of the equation into two parts.
\by
&&\underset{ ||\theta_2||\leq C\sqrt{d_n}}{\sup}\frac{1}{n}\sumn | \I (\epsilon_i < u_{ni} + \tilde{\W}(\z_i)^{'}\theta_2)-\I (\epsilon_i < 0) | \\
&\leq&\underset{ ||\theta_2||\leq C\sqrt{d_n}}{\sup}\frac{1}{n}\sumn | \I (\epsilon_i < u_{ni} + \tilde{\W}(\z_i)^{'}\theta_2)-\I (\epsilon_i < 0)  -\P (\epsilon_i < u_{ni} + \tilde{\W}(\z_i)^{'}\theta_2)+\P (\epsilon_i < 0 )|\\
&& +\underset{||\theta_2||\leq C\sqrt{d_n}}{\sup}\frac{1}{n}\sumn|\P (\epsilon_i < u_{ni} +\tilde{\W}(\z_i)^{'}\theta_2)-\P (\epsilon_i < 0 )|\\
&=& I_1 + I_2
\ey
Consider $I_2$,
 \by
 I_2  &=& \underset{ ||\theta_2||\leq C\sqrt{d_n}}{\sup}\frac{1}{n}\sumn|F_i(u_{ni} +\tilde{\W}(\z_i)^{'}\theta_2)-F_i ( 0 )|\\
 &\leq& \underset{||\theta_2||\leq C\sqrt{d_n}}{\sup}\frac{C}{n}\sumn|u_{ni} +\tilde{\W}(\z_i)^{'}\theta_2|\\
 &\leq& C\underset{||\theta_2||\leq\sqrt{d_n}}{\sup}  \max_i|u_{ni}| +\max_i\|\tilde{\W}(\z_i)^{'}\|\cdot\|\theta_2\| \\
 &=& O_p(k_n^{-r}+\sqrt{\frac{k_nd_n}{n}}) = o_p(1)
 \ey
 where the second last equation follows from {Lemma \ref{Prop of Spline}} and  $\max_i|u_{ni}|=O(k_n^{-r})$.

 As for $I_1$, set for $v_i,~i = 1,\ldots,n$ as
 \by
 v_i = \I (\epsilon_i < u_{ni} + \tilde{\W}(\z_i)^{'}\theta_2)-\I (\epsilon_i < 0)  -\P (\epsilon_i < u_{ni} + \tilde{\W}(\z_i)^{'}\theta_2)+\P (\epsilon_i < 0 ),
 \ey
 which are independent mean-zero random variables satisfying $|v_i| \leq 2$.
 Note that $\var (v_i) = \ep[\I (\epsilon_i < u_{ni} + \tilde{\W}(\z_i)^{'}\theta_2)-\I (\epsilon_i < 0) ]^2$ and $\I (\epsilon_i < u_{ni} + \tilde{\W}(\z_i)^{'}\theta_2)-\I (\epsilon_i < 0)$ is nonzero only when $0<\epsilon_i< u_{ni}+\tilde{\W}(\z_i)^{'}\theta_2$ or $0>\epsilon_i>u_{ni}+\tilde{\W}(\z_i)^{'}\theta_2$, depending on the sign of $u_{ni}+\tilde{\W}(\z_i)^{'}\theta_2$.
 Thus, under the set $\{\theta_2: ||\theta_2||\leq C\sqrt{d_n}\}$
 \by
 \sumn\var(v_i) &=& \sumn\ep(\I (\epsilon_i < u_{ni} + \tilde{\W}(\z_i)^{'}\theta_2)-\I (\epsilon_i < 0))^2\\
 &\leq& \sumn \P\Big(0<\epsilon_i< u_{ni}+\tilde{\W}(\z_i)^{'}\theta_2\Big)+\P \Big(0>\epsilon_i>u_{ni}+\tilde{\W}(\z_i)^{'}\theta_2\Big) \\
 &\leq& \P\Big(|\epsilon_i|< |u_{ni}+\tilde{\W}(\z_i)^{'}\theta_2|\Big)\\
 &\leq& \sumn C |u_{ni}+\tilde{\W}(\z_i)^{'}\theta_2| \\
 &=& O_p(nk_n^{-r}+\sqrt{nk_nd_n}),
 \ey
Then apply Bernstein Inequality in Lemma \ref{Bernstein inequality}, for any $\epsilon > 0$
\by
\underset{ ||\theta_2||\leq C\sqrt{d_n}}{\sup}\P(|\sumn v_i|> n\epsilon)&\leq &2\exp(-\frac{n^2\epsilon^2}{2(nk_n^{-r} + \sqrt{nd_nk_n} + 2n\epsilon/3)})\\
&\leq& 2\exp(-{n\epsilon}) \rightarrow 0
\ey
which means $I_1 = o_p(1)$.\qed
\end{proof}

\begin{lemma}\label{L3.22}
Assume Condition 3.1-3.4 hold, then for any finite positive constants M and C,
\by
\underset{||\theta_1- \tilde{\theta}_1|| \leq M, ||\theta_2||\leq C\sqrt{d_n}}{\sup}\Bigg|\frac{1}{n}\sumn \Big(\psi_\alpha(\epsilon_i - u_{ni}  - \tilde{\W}(\z_i)^{'}\theta_2)- \psi_\alpha(\epsilon_i)\Big)\tilde{\x}^{'}_i({\theta}_1-\tilde{\theta}_1)  \Bigg| = o_p(1)
\ey
and
\by
\underset{||\theta_1- \tilde{\theta}_1|| \leq M, ||\theta_2||\leq C\sqrt{d_n}}{\sup}\Bigg|\frac{1}{n}\sumn \Big(\varphi_\alpha(\epsilon_i - u_{ni}  - \tilde{\W}(\z_i)^{'}\theta_2)- \varphi_\alpha(\epsilon_i)\Big)((\tilde{\x}^{'}_i{\theta}_1)^{2} - (\tilde{\x}^{'}_i\tilde{\theta}_1)^{2})\Bigg| = o_p(1).
\ey

\end{lemma}
\begin{proof}
For the first part, notice that
\by
&&\underset{||\theta_1- \tilde{\theta}_1|| \leq M, ||\theta_2||\leq C\sqrt{d_n}}{\sup}\Bigg|\frac{1}{n}\sumn \Big(\psi_\alpha(\epsilon_i - u_{ni}  - \tilde{\W}(\z_i)^{'}\theta_2)- \psi_\alpha(\epsilon_i)\Big)\tilde{\x}^{'}_i({\theta}_1-\tilde{\theta}_1)  \Bigg|\\
&\leq&\underset{||\theta_1- \tilde{\theta}_1|| \leq M, ||\theta_2||\leq C\sqrt{d_n}}{\sup}\frac{1}{n}\sumn \Big|\psi_\alpha(\epsilon_i - u_{ni}  - \tilde{\W}(\z_i)^{'}\theta_2)- \psi_\alpha(\epsilon_i)\Big| \Big|\tilde{\x}^{'}_i({\theta}_1-\tilde{\theta}_1)  \Big|\\
&\leq& \underset{||\theta_1- \tilde{\theta}_1|| \leq M, ||\theta_2||\leq C\sqrt{d_n}}{\sup} \max_i \Big|\psi_\alpha(\epsilon_i - u_{ni}  - \tilde{\W}(\z_i)^{'}\theta_2)- \psi_\alpha(\epsilon_i)\Big|\cdot \max_i ||\tilde{\x}^{'}_i|| \cdot ||{\theta}_1-\tilde{\theta}_1||\\
&\leq&\underset{||\theta_1- \tilde{\theta}_1|| \leq M, ||\theta_2||\leq C\sqrt{d_n}}{\sup} \max_i C\cdot\Big| u_{ni}  +\tilde{\W}(\z_i)^{'}\theta_2\Big|\cdot \max_i ||\tilde{\x}^{'}_i|| \cdot ||{\theta}_1-\tilde{\theta}_1||\\
&\leq& O_p(k_n^{-r} + \sqrt{d_nk_n/n})\cdot O_p(\sqrt{q_n/n}) = o_p(1),
\ey
where the last inequality follows from the fact $\max_i||\tilde{\x}_i|| = O(\sqrt{q_n/n})$.

Consider the second part
\by
&&\underset{||\theta_1- \tilde{\theta}_1|| \leq M, ||\theta_2||\leq C\sqrt{d_n}}{\sup}\Bigg|\frac{1}{n}\sumn \Big(\varphi_\alpha(\epsilon_i - u_{ni}  - \tilde{\W}(\z_i)^{'}\theta_2)- \varphi_\alpha(\epsilon_i)\Big)((\tilde{\x}^{'}_i{\theta}_1)^{2} - (\tilde{\x}^{'}_i\tilde{\theta}_1)^{2})\Bigg|\\
&\leq& \underset{||\theta_1- \tilde{\theta}_1|| \leq M, ||\theta_2||\leq C\sqrt{d_n}}{\sup}\frac{2}{n}\sumn\Big |\I (\epsilon_i < u_{ni} + \tilde{\W}(\z_i)^{'}\theta_2)-\I (\epsilon_i < 0)\Big| \cdot\max_{i}\Big|(\tilde{\x}^{'}_i{\theta}_1)^{2} - (\tilde{\x}^{'}_i\tilde{\theta}_1)^{2}\Big|
\ey
Notice that $\max_i||\tilde{\x}_i|| = O(\sqrt{q_n/n})$ and use Lemma {\ref{ShL5}},
\by
\max_{i}\Big|((\tilde{\x}^{'}_i{\theta}_1)^{2} - (\tilde{\x}^{'}_i\tilde{\theta}_1)^{2})\Big|
&=&\max_i\Big|(\theta_1+\tilde{\theta}_1)^{'}\tilde{x}_i\tilde{x}_i^{'}(\theta_1-\tilde{\theta}_1) \Big|\\
&\leq&  \max_{i} ||\tilde{x}^{}_i||^2\cdot||\theta_1-\tilde{\theta}_1||\cdot||\theta_1+\tilde{\theta}_1||\\
&\leq&  \max_{i} ||\tilde{x}^{}_i||^2\cdot||\theta_1-\tilde{\theta}_1||\cdot\big(||\theta_1-\tilde{\theta}_1|| + 2||\tilde{\theta}_1||\big)\\
&\leq&C\cdot\frac{q_n^{3/2}}{n} = o_p(1)
\ey
Thus, apply Lemma {\ref{L3.21}}, we have
\by
\underset{||\theta_1- \tilde{\theta}_1|| \leq M, ||\theta_2||\leq C\sqrt{d_n}}{\sup}\Bigg|\frac{1}{n}\sumn \Big(\varphi_\alpha(\epsilon_i - u_{ni}  - \tilde{\W}(\z_i)^{'}\theta_2)- \varphi_\alpha(\epsilon_i)\Big)((\tilde{\x}^{'}_i{\theta}_1)^{2} - (\tilde{\x}^{'}_i\tilde{\theta}_1)^{2})\Bigg| = o_p(1)
\ey\qed
\end{proof}

\begin{lemma}\label{ShL6}
Assume Condition 3.1-3.4 hold, then
\by
||\hat{\theta}_1 - \tilde{\theta}_1|| = o_p(1).
\ey

\end{lemma}
\begin{proof}
Define
\by
\tilde{Q}_i(\theta_1,\tilde{\theta}_1,\theta_2) = \phi_\alpha(\epsilon_i - u_{ni} -\tilde{\x}_i^{'}\theta_1 - \tilde{\W}(\z_i)^{'}\theta_2)  - \phi_\alpha(\epsilon_i - u_{ni} -\tilde{\x}_i^{'}\tilde{\theta}_1 - \tilde{\W}(\z_i)^{'}\theta_2).
\ey
We first show that for any positive constants $C$ and $M$,
\be\label{Euq6.9}
\P\left(\underset{||\theta_1- \tilde{\theta}_1|| >M, ||\theta_2||\leq C\sqrt{d_n}}{\inf}\frac{1}{n}\sumn \tilde{Q}_i(\theta_1,\tilde{\theta}_1,\theta_2)>0 \right) \rightarrow 1.
\ee
Notice
\by
\tilde{Q}_i(\theta_1,\tilde{\theta}_1,\theta_2) = (\tilde{Q}_i(\theta_1,\tilde{\theta}_1,\theta_2) - \ep[\tilde{Q}_i(\theta_1,\tilde{\theta}_1,\theta_2)]) + \ep[\tilde{Q}_i(\theta_1,\tilde{\theta}_1,\theta_2)].
\ey
Considering the expectation part, we are going to show that
\by
\underset{||\theta_1- \tilde{\theta}_1|| \leq M, ||\theta_2||\leq C\sqrt{d_n}}{\sup} |\frac{1}{n}\sumn \ep[\tilde{Q}_i(\theta_1,\tilde{\theta}_1,\theta_2)] - \frac{1}{2n}\Big(\theta^{'}_1W_n\theta_1- \tilde{\theta}^{'}_1W_n\tilde{\theta}_1\Big)| = o_p(1).
\ey
 Through Taylor Expansion,
\by
&&\frac{1}{n}\sumn \ep[\tilde{Q}_i(\theta_1,\tilde{\theta}_1,\theta_2)] \\
&=&\frac{1}{n}\sumn \ep\Big[\phi_\alpha(\epsilon_i - u_{ni}  - \tilde{\W}(\z_i)^{'}\theta_2) -
\psi_\alpha(\epsilon_i - u_{ni}  - \tilde{\W}(\z_i)^{'}\theta_2)\tilde{\x}^{'}_i{\theta}_1\\
&&+\frac{1}{2}\varphi_\alpha(\epsilon_i - u_{ni}  - \tilde{\W}(\z_i)^{'}\theta_2)(\tilde{\x}^{'}_i{\theta}_1)^{2} + o(\tilde{\x}^{'}_i{\theta}_1)^{2}\\
&& - \phi_\alpha(\epsilon_i - u_{ni} - \tilde{\W}(\z_i)^{'}\theta_2)] -\psi_\alpha(\epsilon_i - u_{ni}  - \tilde{\W}(\z_i)^{'}\theta_2)\tilde{\x}^{'}_i\tilde{\theta}_1\\
&&+\frac{1}{2}\varphi_\alpha(\epsilon_i - u_{ni}  - \tilde{\W}(\z_i)^{'}\theta_2)(\tilde{\x}^{'}_i\tilde{\theta}_1)^{2} + o(\tilde{\x}^{'}_i\tilde{\theta}_1)^{2}\Big]\\
&=& \frac{1}{n}\sumn \ep\Big[ -
\psi_\alpha(\epsilon_i - u_{ni}  - \tilde{\W}(\z_i)^{'}\theta_2)\tilde{\x}^{'}_i({\theta}_1-\tilde{\theta}_1)\\
&&+\frac{1}{2}\varphi_\alpha(\epsilon_i - u_{ni}  - \tilde{\W}(\z_i)^{'}\theta_2)((\tilde{\x}^{'}_i{\theta}_1)^{2} - (\tilde{\x}^{'}_i\tilde{\theta}_1)^{2})(1+o(1))\Big]
\ey
Then applying Lemma 6.8, we have
\by
&&\frac{1}{n}\sumn \ep[\tilde{Q}_i(\theta_1,\tilde{\theta}_1,\theta_2)] \\
&=&\frac{1}{n}\sumn \ep\Big[ -
\psi_\alpha(\epsilon_i )\tilde{\x}^{'}_i({\theta}_1-\tilde{\theta}_1)\Big](1+o_p(1))\\
&&+\frac{1}{n}\sumn \ep\Big[\frac{1}{2}\varphi_\alpha(\epsilon_i )((\tilde{\x}^{'}_i{\theta}_1)^{2} - (\tilde{\x}^{'}_i\tilde{\theta}_1)^{2})\Big](1+o_p(1))\\
&=&\frac{1}{n}\sumn \ep\Big[ \frac{1}{2}\varphi_\alpha(\epsilon_i )((\tilde{\x}^{'}_i{\theta}_1)^{2} - (\tilde{\x}^{'}_i\tilde{\theta}_1)^{2})\Big](1+o_p(1))
\ey
under the set $\{(\theta_1,\theta_2):||\theta_1- \tilde{\theta}_1|| \leq M, ||\theta_2||\leq C\sqrt{d_n}\}$, for any positive constants M and C, where the last equation follows the fact that $\ep[\phi_\alpha(\epsilon_i)] = 0,~~\forall i =1,\ldots,n$.  And this implies
\by
\underset{||\theta_1- \tilde{\theta}_1|| \leq M, ||\theta_2||\leq C\sqrt{d_n}}{\sup} |\frac{1}{n}\sumn \ep[\tilde{Q}_i(\theta_1,\tilde{\theta}_1,\theta_2)] - \frac{1}{2n}\Big(\theta^{'}_1W_n\theta_1- \tilde{\theta}^{'}_1W_n\tilde{\theta}_1\Big)| = o_p(1).
\ey
Next we introduce
\by
R_{i,n}(\theta_1) &=&  \phi_\alpha(\epsilon_i - u_{ni} -\tilde{\x}_i^{'}\theta_1 - \tilde{\W}(\z_i)^{'}\theta_2) -  \phi_\alpha(\epsilon_i - u_{ni}  - \tilde{\W}(\z_i)^{'}\theta_2)\\
&& +  \psi_\alpha(\epsilon_i - u_{ni} - \tilde{\W}(\z_i)^{'}\theta_2)\tilde{\x}_i^{'}\theta_1.
\ey
Through Taylor Expansion, there exits $0<\xi_{i,\theta_1},~\xi_{i,\tilde{\theta}_1}<1$ satisfying
\by
R_{i,n}({\theta}_1) &= &\frac{1}{2} \varphi_\alpha(\epsilon_i - u_{ni} -\xi_{i,\theta_1}\tilde{\x}_i^{'}\theta_1 - \tilde{\W}(\z_i)^{'}\theta_2)(\tilde{\x}_i^{'}\theta_1)^2\\
R_{i,n}({\tilde{\theta}}_1) &= &\frac{1}{2} \varphi_\alpha(\epsilon_i - u_{ni} -\xi_{i,\tilde{\theta}}\tilde{\x}_i^{'}\tilde{\theta} - \tilde{\W}(\z_i)^{'}\theta_2)(\tilde{\x}_i^{'}\tilde{\theta}_1)^2.
\ey
Notice $\frac{1}{2} \varphi_\alpha(\epsilon_i - u_{ni} -\xi_{i,\theta_1}\tilde{\x}_i^{'}\theta_1 - \tilde{\W}(\z_i)^{'}\theta_2) \leq 1$, thus
\by
&&\underset{||\theta_1- \tilde{\theta}_1|| \leq M, ||\theta_2||\leq C\sqrt{d_n}}{\sup}\Big|\frac{1}{n}\sumn R_{i,n}(\theta_1) -  R_{i,n}(\tilde{\theta}_1) -\ep[ R_{i,n}(\theta_1) -  R_{i,n}(\tilde{\theta}_1)]  \Big|\\
&\leq& 2\underset{||\theta_1- \tilde{\theta}_1|| \leq M, ||\theta_2||\leq C\sqrt{d_n}}{\sup}\Big|\frac{1}{n}\sumn((\tilde{\x}_i^{'}\theta_1)^2 - (\tilde{\x}_i^{'}\tilde{\theta}_1)^2)\Big|\\
&\leq&C\underset{||\theta_1- \tilde{\theta}_1|| \leq M, ||\theta_2||\leq C\sqrt{d_n}}{\sup}\max_{i} \Big|2(\theta_1-\tilde{\theta}_1)^{}\tilde{\x}_i\tilde{\x}^{'}_i\tilde{\theta}_1\Big|\\
&\leq& C\underset{||\theta_1- \tilde{\theta}_1|| \leq M, ||\theta_2||\leq C\sqrt{d_n}}{\sup} \max_{i} ||\tilde{x}^{}_i||^2\cdot||\theta_1-\tilde{\theta}_1||\cdot||\tilde{\theta}_1||\\
&\leq&C\frac{q_n^{3/2}}{n} = o_p(1)
\ey
where last inequality follows from the fact $\max_i||\tilde{\x}_i|| = O(\sqrt{q_n/n})$ and Lemma {\ref{ShL5}}.

Hence, under the set $\{(\theta_1,\theta_2):{||\theta_1- \tilde{\theta}_1|| \leq M, ||\theta_2||\leq C\sqrt{d_n}}\}$
\by
&&\frac{1}{n}\sumn \tilde{Q}_i(\theta_1,\tilde{\theta}_1,\theta_2) -\frac{1}{n}\sumn \ep[\tilde{Q}_i(\theta_1,\tilde{\theta}_1,\theta_2)] \\
&=& \frac{1}{n}\sumn\Bigg(-\psi_\alpha(\epsilon_i - u_{ni} - \tilde{\W}(\z_i)^{'}\theta_2)\tilde{\x}_i^{'}(\theta_1-\tilde{\theta}_1)  + R_{i,n}(\theta_1) -  R_{i,n}(\tilde{\theta}_1)\\
&&+ \ep[\psi_\alpha(\epsilon_i - u_{ni} - \tilde{\W}(\z_i)^{'}\theta_2)\tilde{\x}_i^{'}(\theta_1-\tilde{\theta}_1)\  -\ep[ R_{i,n}(\theta_1) -  R_{i,n}(\tilde{\theta}_1)]
\Bigg)\\
&=& \frac{1}{n}\sumn\Bigg(-\psi_\alpha(\epsilon_i )\tilde{\x}_i^{'}(\theta_1-\tilde{\theta}_1)(1+o_p(1))  + R_{i,n}(\theta_1) -  R_{i,n}(\tilde{\theta}_1)\\
&&+ \ep[\psi_\alpha(\epsilon_i )\tilde{\x}_i^{'}(\theta_1-\tilde{\theta}_1)(1+o_p(1))  -\ep[ R_{i,n}(\theta_1) -  R_{i,n}(\tilde{\theta}_1)]\Bigg)\\
&=&\frac{1}{n}\sumn\Bigg(-\psi_\alpha(\epsilon_i )\tilde{\x}_i^{'}(\theta_1-\tilde{\theta}_1)(1+o_p(1))  \\
&& + R_{i,n}(\theta_1) -  R_{i,n}(\tilde{\theta}_1) -\ep[ R_{i,n}(\theta_1) -  R_{i,n}(\tilde{\theta}_1)]\Bigg)
\ey
where the second last equality uses Lemma \ref{L3.22} and the last equality follows by the fact $\ep[\psi_\alpha(\epsilon_i)] = 0,~\forall i = 1,\ldots,n.$ And this shows
\by\underset{||\theta_1- \tilde{\theta}_1|| \leq M, ||\theta_2||\leq C\sqrt{d_n}}{\sup}\Bigg|\frac{1}{n}\sumn\Big( \tilde{Q}_i(\theta_1,\tilde{\theta}_1,\theta_2) - \ep[\tilde{Q}_i(\theta_1,\tilde{\theta}_1,\theta_2)] \\+\psi_\alpha(\epsilon_i )\tilde{\x}_i^{'}(\theta_1-\tilde{\theta}_1) \Big) \Bigg| = o_p(1).\ey

Notice the definition of $\tilde{\theta}_1$ and Lemma \ref{Prop of Design}, it follows that
\by
\sumn \psi_\alpha(\epsilon_i )\tilde{\x}_i^{'}(\theta_1-\tilde{\theta}_1) &=& (\theta_1-\tilde{\theta}_1)^{'}n^{-1/2}X^{*'}\psi_\alpha(\epsilon)\\
&=& (\theta_1-\tilde{\theta}_1)^{'}W_n\tilde{\theta}_1(1+o_p(1)).
\ey
Then, by using the result
 \by
\underset{||\theta_1- \tilde{\theta}_1|| \leq M, ||\theta_2||\leq C\sqrt{d_n}}{\sup} \Big|\frac{1}{n}\sumn \ep[\tilde{Q}_i(\theta_1,\tilde{\theta}_1,\theta_2)] - \frac{1}{2n}\Big(\theta^{'}_1W_n\theta_1- \tilde{\theta}^{'}_1W_n\tilde{\theta}_1\Big)\Big| = o_p(1),
\ey
we have
\by
\underset{||\theta_1- \tilde{\theta}_1|| \leq M, ||\theta_2||\leq C\sqrt{d_n}}{\sup} \Big|\frac{1}{n}\sumn \tilde{Q}_i(\theta_1,\tilde{\theta}_1,\theta_2) - \frac{1}{2n}\Big(\theta^{'}_1W_n\theta_1- \tilde{\theta}^{'}_1W_n\tilde{\theta}_1\Big)\\+ \frac{1}{n}(\theta_1-\tilde{\theta}_1)^{'}W_n\tilde{\theta}_1\Big|  = o_p(1),\\
\ey
which means
\by
\underset{||\theta_1- \tilde{\theta}_1|| \leq M, ||\theta_2||\leq C\sqrt{d_n}}{\sup} \Big|\frac{1}{n}\sumn \tilde{Q}_i(\theta_1,\tilde{\theta}_1,\theta_2) - \frac{1}{2n}(\theta_1-\tilde{\theta}_1)^{'}W_n(\theta_1-\tilde{\theta}_1)\Big|
= o_p(1).
\ey
Thus from Condition \ref{C(2)}, for any $\theta_1$ satisfying $||\theta_1- \tilde{\theta}_1|| \geq M>0$, we have
\by
 \frac{1}{2n}(\theta_1-\tilde{\theta}_1)^{'}W_n(\theta_1-\tilde{\theta}_1)> 0,
\ey
which implies (\ref{Euq6.9}) holds. Thus, the result follows.\qed
\end{proof}
\proofthm3

\bn[(1)]
\item For the first part (11), from the results of Lemma \ref{ShL5} and Lemma \ref{ShL6}, we can show $||\hat{\c}^*_A - \bbeta^*_A|| = O_p(\sqrt{q_n/n})$. Hence through the fact that $\hat{c}^*_A = \hat{\bbeta}^*_A $, our first result can be proved.
\item For the second part (3.4), we have  shown that $\hat{g}(\z_i) = \tilde{g}(\z_i)$, thus by Lemma \ref{ShL4}, the second result holds.\qed
\en

\begin{lemma}\label{Sufficient condition for local optimal minimizer}
	Suppose the objective function $L(\btheta)$ can be decomposed as the difference of two convex functions $k(\btheta)$ and $l(\btheta)$, i.e.,
	\bqs L(\btheta)=k(\btheta)-l(\btheta),\eqs
	with the corresponding subdifferential functions $\partial k(\btheta)$ and $\partial l(\btheta)$ respectively. Let $\text{dom}(k)=\{\btheta:k(\btheta)<\infty\}$ be the effective domain of $k(\btheta)$ and $\btheta^*$ be a point that has neighbourhood $U$ such that $\partial l(\btheta)\cap \partial k(\btheta^*)\neq \emptyset$,  $\forall~ \btheta \in U\cap \text{dom}(k)$. Then $\btheta^*$ is  a local minimizer of $f(\btheta)$.
\end{lemma}
\proof The proof is available in \cite{TaA97}.\qed
\begin{lemma}\label{momentsbounding}
Denote $\{X_i\}^n_{i=1}$ a sequence of independent real valued random variables with $\ep X_i = 0$ and $S_n = \sum_{i=1}^n X_i$. Then for $r \geq 2$, the following inequality holds:
\by
\ep |S_n|^r \leq C_rn^{r/2-1}\sum_{i=1}^{n}\ep |X_i|^r,
\ey
where $C_r$ is some constant only depending on r.
\end{lemma}
\proof{Details can be found in \cite{Chung08}}.\qed

\begin{lemma}\label{Subgradient of oracle estimator}
	Assume Conditions 3.1-3.5p are satisfied. The tuning parameter $\lambda = o\left(n^{-(1-  C_4)/2}\right)$, $q_n = o(n\lambda^2)$, $k_n = o(n\lambda^2)$ and $p = o\big((n\lambda^2)^k\big)$. For the oracle estimator $(\hat{\bbeta}^*,\hat{\bxi}^*)$, with probability tending to one,
	\be
	s_j(\hat{\bbeta}^*,\hat{\bxi}^*)&=&0,~~j=1,\ldots,q_n ~\text{or}~ j=p+1,\ldots,p+D_n,\\
	|\hat{\beta}^*_j|&\geq& (a+1/2)\lambda,~~j=1,\ldots,q_n,\\
	|s_j(\hat{\bbeta}^*,\hat{\bxi}^*)|&\leq&\lambda,~~j=q_n+1,\ldots,p.
	\ee
\end{lemma}

\begin{proof}
\bn[(1)]
\item Proof for (6.5).  For $j=1,\ldots,q_n ~\text{or}~ j=p+1,\ldots,p+D_n$, by the first order necessary condition of the optimal solution,
\by
s_{j}(\hat{\bbeta}^*,\hat{\bxi}^*) = \frac{\partial}{\partial\beta_j}\Big(\frac{1}{n}\sum_{i=1}^{n}\phi_{\alpha}(y_i-x_i^{'}\bbeta-\bPi(z_i)^{'}\bxi)\Big)\mid_{(\bbeta,\bxi)=(\hat{\bbeta}^*,\hat{\bxi}^*)} = 0.
\ey

\item Proof for Inequality (6.6). It's sufficient to prove that
\by
\P\left(\underset{1 \leq j \leq q_n}{\min}|\hat{\beta}^*_j| \geq (a+1/2)\lambda)\right)\rightarrow1, \text{as}~~ n \rightarrow \infty.
\ey
Notice that
\by\underset{1 \leq j \leq q_n}{\min}|\hat{\beta}^*_j|  \geq \underset{1 \leq j \leq q_n}{\min}|\beta_j^*| - \underset{1 \leq j \leq q_n}{\max}|\beta_j^* - \hat{\beta}^*_j|.\ey
By Theorem 3.1 and Condition \ref{C(4)}, $\|\hat{\bbeta}_{A}^* - {\bbeta}_A^*\| =O_p(\sqrt{\frac{q_n}{n}}) =O_p(n^{-(1-C_3)/2})$, then,
\by
\underset{1 \leq j \leq q_n}{\max}|\beta_j^* - \hat{\beta}^*_j| =  O_p(n^{-(1-C_3)/2}) = o_p(n^{-(1-C_4)/2}).
\ey
Also, Condition \ref{C(5)} shows $n^{(1-C_4)/2}\underset{1\leq j \leq q_n}{\min}|\bbeta_{j}^*| \geq C_5$. Thus, inequality (24) can hold by setting $\lambda = o(n^{-(1-C_4)/2})$.

\item Proof for Inequality (6.7). For $j = q_n+1,\ldots,p$, recall the definition of $s_j(\hat{\bbeta}^*,\hat{\bxi}^*)$ as
\by
s_{j}(\hat{\bbeta}^*,\hat{\bxi}^*)&=& \frac{\partial}{\partial\beta_j}\Big(\frac{1}{n}\sum_{i=1}^{n}\phi_{\alpha}(y_i-x_i^{'}\bbeta-\bPi(z_i)^{'}\bxi\Big)\mid_{(\bbeta,\bxi)=(\hat{\bbeta}^*,\hat{\bxi}^*)}\\
                                                                    &=& -\frac{2}{n}\sum_{i =1}^{n}x_{ij}(y_i-x^{'}_i\hat{\bbeta}^*-\bPi(z_i)^{'}\hat{\bxi}^*)|\I(y_i-x^{'}_i\hat{\bbeta}^*-\bPi(z_i)^{'}\hat{\bxi}^* < 0)-\alpha|\\
                                                                    &=& -\frac{2}{n}\sum_{i =1}^{n}x_{ij}(y_i-x^{'}_{Ai}\hat{\bbeta}_{A}^*-\bPi(z_i)^{'}\hat{\bxi}^*)|\I(y_i-x^{'}_{Ai}\hat{\bbeta}_{A}^*-\bPi(z_i)^{'}\hat{\bxi}^* < 0)-\alpha|,
\ey
where the last equality follows by the definition of oracle estimator $(\obeta,\oxi)$.

To prove the result, first we need to show that
\by
\P\left(\underset{q_{n}+1\leq j \leq p}{\max}{\Big|}\frac{2}{n}\sum_{i =1}^{n}x_{ij}(y_i-x^{'}_i\hat{\bbeta}^*-\bPi(z_i)^{'}\hat{\bxi}^*)|\I(y_i-x^{'}_i\hat{\bbeta}^*-\bPi(z_i)^{'}\hat{\bxi}^* < 0)-\alpha|{\Big|}>\lambda\right)\rightarrow 0,
\ey
which is equivalent to show that
\by
\P\left(\underset{q_{n+1}\leq j \leq p}{\max}{\Big|}\frac{2}{n}\sum_{i =1}^{n}x_{ij}(y_i-x^{'}_i\hat{\bbeta}^*-\bW(z_i)^{'}\hat{\bgamma})|\I(y_i-x^{'}_i\hat{\bbeta}^*-\bW(z_i)^{'}\hat{\bgamma} < 0)-\alpha|{\Big|}>\lambda\right)\rightarrow 0.
\ey
Set \by
\begin{split}
&\epsilon_i(\bbeta_A,\bgamma) = y_i - x_{Ai}^{'}\bbeta_{A}-\bW(z_i)^{'}\bgamma,\\
& \hat{\epsilon}_i = \epsilon_i(\hat{\bbeta}_A^*,\hat{\bgamma}) =  y_i - x_{Ai}^{'}\hat{\bbeta}^*_A-\bW(z_i)^{'}\hat{\bgamma} = y_i - x_{Ai}^{'}\hat{\bbeta}^*_A-\hat{g}(z_i) ,\\
&\epsilon^*_i = y_i - x_{Ai}^{'}\bbeta^*_{A}-g_0(z_i).
\end{split}
\ey
We also set \by
I_j =\frac{2}{n}\sum_{i =1}^{n}x_{ij}\hat{\epsilon}_i|\I(\hat{\epsilon}_i\leq0)-\alpha|
=I_{j1}+I_{j2} ,
\ey
with $I_{j1}=\frac{2}{n}\sum_{i =1}^{n}x_{ij}(\hat{\epsilon}_i - \epsilon^*_i)|\I(\hat{\epsilon}_i\leq0)-\alpha|$, and $I_{j2}=\frac{2}{n}\sum_{i =1}^{n}x_{ij}\epsilon^*_i|\I(\hat{\epsilon}_i\leq0)-\alpha|$.

Let's first consider $\P\left(\underset{q_{n+1}\leq j \leq p}{\max}|I_{j1}|>\lambda/2\right)\rightarrow0$. It follows from the proof of Lemma \ref{ShL4} that$ \|\bgamma_0-\hat{\bgamma}\|_2=O_p(\sqrt{\frac{k_nd_n}{n}})$.
Note that
\by
\frac{1}{n}\sum_{i=1}^{n}|\hat{\epsilon}_i - \epsilon^*_i| &=& \frac{1}{n}\sumn|x_{Ai}^{'}(\hat{\bbeta}^*_A-\bbeta_A^*)+ \bW(z_i)^{'}(\hat{\bgamma}-\bgamma_0)+u_{ni}|\\
&\leq&\frac{1}{n}\sum_{i=1}^{n}(|x_{Ai}^{'}(\hat{\bbeta}^*_A-\bbeta_A^*)| + |\bW(z_i)^{'}(\hat{\bgamma}-\bgamma_0)|+|u_{ni}|)\\
&\leq&(\frac{1}{n}\sum_{i=1}^{n}{|x_{Ai}^{'}(\hat{\bbeta}^*_A-\bbeta_A^*)|}^2)^{1/2} + (\frac{1}{n}\sum_{i=1}^{n}{|\bW(z_i)^{'}(\hat{\bgamma}-\bgamma_0)| }^2)^{1/2} + \underset{1\leq i\leq n}{\sup}|u_{ni}|\\
&\leq& \lambda_{max}^{1/2}(\frac{1}{n}X_AX_A^{'})\|\bbeta_A^*-\hat{\bbeta}^*_A\|_2 + (\frac{1}{n}(\hat{\bgamma}-\bgamma_0)^{'}\bW^2(\hat{\bgamma}-\bgamma_0))^{1/2} + \underset{1\leq i\leq n}{\sup}|u_{ni}|\\
&=&O_p(\sqrt{q_n/n}+\sqrt{d_n/n}+k_n^{-r}).
\ey
where the second inequality follows Jensen inequality and the last inequality applies Condition 3.2.

Thus, by using Condition 3.2 again and Condition 3.3 and 3.4,
\by
\underset{q_{n+1}\leq j \leq p}{\max}|I_{j1}| &\leq&  \ 2 \cdot \underset{q_{n+1}\leq j \leq p}{\max}|x_{ij}|\cdot \frac{1}{n}\sum_{i=1}^{n}|\hat{\epsilon}_i - \epsilon^*_i|\\
 &\leq& C\cdot \frac{1}{n}\sum_{i=1}^{n}|\hat{\epsilon}_i - \epsilon^*_i|\\
&= &O_p(\sqrt{q_n/n}+\sqrt{d_n/n}+k_n^{-r}) =o_p(\lambda),
\ey
which means,
\by
\P\left(\underset{q_{n+1}\leq j \leq p}{\max}|I_{j1}|>\lambda/2\right)\rightarrow0.
\ey
As for $I_{j2}$,
\by
I_{j2} &\leq& \frac{2}{n}\sumn x_{ij}\epsilon_i^*|\I(\hat{\epsilon}_i\leq0)-\I(\epsilon^*_i\leq0)| + \frac{2}{n}\sumn x_{ij}\epsilon^*_i|\alpha-\I(\epsilon_i^*\leq0)| \\
  &=&I_{j21}+I_{j22}.
\ey
To evaluate $I_{j21}$, note that
\bqs |\I(\hat{\epsilon}_i\leq0)-\I(\epsilon^*_i\leq0)| = |\I(\epsilon^*_i\leq\epsilon^*_i-\hat{\epsilon}_i)-\I(\epsilon^*_i\leq0)|\leq \I (|\epsilon^*_i|\leq|\epsilon^*_i-\hat{\epsilon}_i|).\eqs
So, we have for $j=q_{n+1},\ldots,p$,
\by
\underset{q_{n+1}\leq j \leq p}{\max}|I_{j21}|  &\leq& \underset{q_{n+1}\leq j \leq p}{\max} 2\times\frac{1}{n}\sumn|x_{ij}|\times|\epsilon^*_i|\times\I (|\epsilon^*_i|\leq|\epsilon^*_i-\hat{\epsilon}_i|)\\
&\leq& \underset{q_{n+1}\leq j \leq p}{\max} C\times\frac{1}{n}\sumn|\epsilon^*_i-\hat{\epsilon}_i|\\
&= &O_p(\sqrt{q_n/n}+\sqrt{d_n/n}+k_n^{-r}) =o_p(\lambda).
\ey
Thus, we can show that $\P\left(\underset{q_{n+1}\leq j \leq p}{\max}|I_{j21}|>\lambda/4\right)\rightarrow0$.

Now Consider $I_{j22}$, we define $\eta_i = \epsilon^*_i|\I(\epsilon^*_i \leq 0) - \alpha|$, which is independent and satisfies $\ep \left(\eta_i|x_i \right) = 0$. By Condition \ref{C(1)} we have $\ep\left(\eta_i^{2k}|x_i\right) < \infty$. Also, $x_{ij}$is bounded because of  Condition \ref{C(2)}, thus by Lemma \ref{momentsbounding}, we have $\ep I^{2k}_{j22}= O(n^{-k})$. Therefore, by Markov Inequality, we have
\by
\P\left(|I_{j22}|>\lambda\right)\leq \frac{\ep I^{2k}_{j22}}{(\lambda^{2k})}=O((n\lambda^2)^{-k}),
\ey
which contains
\by
\P\left(\underset{q_n+1\leq j \leq p}{\max}|I_{j22}|>\lambda/4\right)\leq \sum^{p}_{q_n+1}\P\left(|I_{j22}|>\lambda/4\right)=O(p(n\lambda^2)^{-k}) \rightarrow 0.
\ey
Overall,
\by
&&\P\left(\underset{q_n+1\leq j \leq p}{\max}|I_{j}|>\lambda\right)\\
&\leq&\P\left(\underset{q_n+1\leq j \leq p}{\max}|I_{j1}|>\lambda/2\right) + \P\left(\underset{q_n+1\leq j \leq p}{\max}|I_{j2}|>\lambda/2\right)\\
 &\leq& \P\left(\underset{q_n+1\leq j \leq p}{\max}|I_{j1}|>\lambda/2\right) + \P\left(\underset{q_n+1\leq j \leq p}{\max}|I_{j22}|>\lambda/4\right) \\
 &&+~ \P\left(\underset{q_n+1\leq j \leq p}{\max}|I_{j21}|>\lambda/4\right) \rightarrow 0
\ey\qed
\en
\end{proof}
\noindent{\bf Proof of Theorem 3.2.}

 We take use of Lemma \ref{Sufficient condition for local optimal minimizer} to prove our theorem. Consider any $(\bbeta',\bxi')'$ in a ball $\mathcal{B}(\lambda)$ with the center $(\hat{\bbeta}^*,\hat{\bxi}^*)$ and radius $\lambda/2$. It is sufficient to show that for any $(\bbeta',\bxi')'\in \mathcal{B}(\lambda)$, with probability tending to one,
\by
\partial l(\bbeta,\bxi) \cap \partial k(\hat{\bbeta}^*,\hat{\bxi}^*)\neq \emptyset.
\ey

Define the event $\mathcal{E}_1=\{|\hat{\beta}_j^*|\geq(a+1/2)\lambda,~1\leq j\leq q_n\}$. Then by Lemma \ref{Subgradient of oracle estimator}, $\P(\mathcal{E}_1)\to 1$, as $n\to \infty$.
For $j=1,\ldots,q_n$, on the event $\mathcal{E}_1$, for any $(\bbeta',\bxi')'\in \mathcal{B}(\lambda)$,
\by
\underset{1\leq j\leq q_n}{\min}{|\beta_j|}\geq \underset{1\leq j\leq q_n}{\min}{|\hat{\beta}_j^*|}-\underset{1\leq j\leq q_n}{\max}{|\hat{\beta}_j^*-\beta_j|}\geq (a+1/2)\lambda-\lambda/2=a\lambda.
\ey
So $H_{\lambda}'(\beta_j)=\lambda\text{sgn}(\beta_j)$, i.e., $\mu_j=\frac{\partial{l(\bbeta,\bxi)}}{\partial \beta_j}=\lambda\text{sgn}(\beta_j)$.
At the same time, $\kappa_j=s_j(\hat{\bbeta}^*,\hat{\bxi}^*)+\lambda l_j,~~\text{for}~j=1,2,\ldots,q_n$. By convex optimization theory or Lemma \ref{Subgradient of oracle estimator}, $s_j(\hat{\bbeta}^*,\hat{\bxi}^*)=0$. Then on the event $\mathcal{E}_1$, if $\text{sgn}(\hat{\beta}_j^*)=\text{sgn}(\beta_j)$, we have that  $\kappa_j=\frac{\partial k(\hat{\bbeta}^*,\hat{\bxi}^*)}{\partial \beta_j}=\mu_j$. In fact, if $\text{sgn}(\hat{\beta}_j^*)\neq \text{sgn}(\beta_j)$, then on the event $\mathcal{E}_1$, $|\hat{\beta}_j^*-\beta_j|=|\hat{\beta_j}^*|+|\beta_j|\geq (a+1/2)\lambda$, which causes contradition with that $(\bbeta',\bxi')'\in \mathcal{B}(\lambda)$.

Define the event $\mathcal{E}_2=\{|s_j(\hat{\bbeta}^*,\hat{\bxi}^*)|\leq \lambda,~q_n+1\leq j\leq p\}$. For $j=q_n+1,\ldots,p$, by the construction of the oracle estimator, we have $\hat{\beta}_j^*=0$. Then for any $(\bbeta',\bxi')'\in \mathcal{B}(\lambda)$,
\by
\underset{q_n+1\leq j\leq p}{\max}{|\beta_j|}\leq \underset{q_n+1\leq j\leq p}{\max}{|\hat{\beta}_j^*|}+\underset{q_n+1\leq j\leq p}{\max}{|\hat{\beta}_j^*-\beta_j|}\leq \lambda/2.
\ey
So in this situation, $\mu_j=\frac{\partial{l(\bbeta,\bxi)}}{\partial \beta_j}=0$.
On the other hand, $\kappa_j=s_j(\hat{\bbeta}^*,\hat{\bxi}^*)+\lambda l_j$ with $l_j\in[-1,1]$. On the event $\mathcal{E}_2=\{|s_j(\hat{\bbeta}^*,\hat{\bxi}^*)|\leq \lambda,~q_n+1\leq j\leq p\}$, there exists $l_j^*,~j=q_n+1,\ldots,p$ such that $
s_j(\hat{\bbeta}^*,\hat{\bxi}^*)+\lambda l_j^*=0=\mu_j$.

For $j=p+l,~l=1,\ldots,D_n$,  by convex optimization theory or Lemma \ref{Subgradient of oracle estimator}, $\kappa_j=s_j(\hat{\bbeta}^*,\hat{\bxi}^*)=0$. Note that
$\mu_j= \frac{\partial l(\bbeta,\bxi)}{\partial \xi_l}=0$. So for $j=p+l,~l=1,\ldots,D_n$, $\kappa_j=\mu_j$.

Combined with all results above, on the event $\mathcal{E}_1\cap \mathcal{E}_2$, we have for any $(\bbeta',\bxi')'\in \mathcal{B}(\lambda)$,
\by
\partial l(\bbeta,\bxi) \cap \partial k(\hat{\bbeta}^*,\hat{\bxi}^*)\neq \emptyset.
\ey
Note that
\by
\P(\mathcal{E}_1\cap \mathcal{E}_2)\geq 1-\P(\bar{\mathcal{E}_1})-\P(\bar{\mathcal{E}_2})\to 1,~\text{as}~n\to\infty.
\ey
So far, the proof has been completed. \qed


\end{document}